\pdfoutput=1
\documentclass[12pt]{amsart}
\usepackage{tabto}
\usepackage{url}
\usepackage{a4wide,bm}
\usepackage{amssymb}
\usepackage{amsthm}
\usepackage{amsmath}
\usepackage{amscd}
\usepackage{verbatim}
\usepackage[all]{xy}
\usepackage{booktabs}
\usepackage{tabularx}
\usepackage{enumitem}
\usepackage{longtable}
\newcommand{\ra}[1]{\renewcommand{\arraystretch}{#1}}
\ra{1.3}
\usepackage{longtable, lscape}
\usepackage{rotating}
\usepackage{multirow}
\usepackage{hyperref}

\textheight8.5in \textwidth6.4in \numberwithin{equation}{section}

\theoremstyle{plain}
\newtheorem{theorem}{Theorem}[section]
\newtheorem{corollary}[theorem]{Corollary}
\newtheorem{lemma}[theorem]{Lemma}
\newtheorem{proposition}[theorem]{Proposition}

\newtheorem*{theorem*}{Theorem}
\theoremstyle{definition}
\newtheorem{definition}[theorem]{Definition}
\newtheorem{example}[theorem]{Example}
\newtheorem{algorithm}[theorem]{Algorithm}

\theoremstyle{remark}
\newtheorem{remark}[theorem]{Remark}

\newcommand{\R}{\mathbb{R}}
\newcommand{\Q}{\mathbb{Q}}
\newcommand{\Z}{\mathbb{Z}}

\newcommand{\B}{\mathbb{B}}
\newcommand{\C}{\mathbb{C}}
\newcommand{\h}{\mathbb{H}}
\renewcommand{\H}{\mathbb{H}}
\renewcommand{\P}{\mathbb{P}}

\newcommand{\zxz}[4]{\begin{pmatrix} #1 & #2 \\ #3 & #4 \end{pmatrix}}

\newcommand{\leg}[2]{\left( \frac{#1}{#2} \right)}
\newcommand{\kzxz}[4]{\left(\begin{smallmatrix} #1 & #2 \\ #3 & #4\end{smallmatrix}\right) }
\newcommand{\kabcd}{\kzxz{a}{b}{c}{d}}

\newcommand{\calA}{\mathcal{A}}
\newcommand{\calB}{\mathcal{B}}
\newcommand{\calC}{\mathcal{C}}

\newcommand{\frake}{\mathfrak e}

\newcommand{\eps}{\varepsilon}
\newcommand{\epsodd}{\varepsilon_{\mathrm{odd}}}
\newcommand{\bs}{\backslash}

\newcommand{\Span}{\operatorname{span}}

\newcommand{\Sl}{\operatorname{SL}}

\newcommand{\SL}{\operatorname{SL}}

\newcommand{\Mp}{\operatorname{Mp}}
\newcommand{\Orth}{\operatorname{O}}

\newcommand{\Aut}{\operatorname{Aut}}

\newcommand{\sig}{\operatorname{sig}}

\newcommand{\Pic}{\operatorname{Pic}}

\newcommand{\abs}[1]{{\lvert{#1}\rvert}}

\newcommand{\sage}{\texttt{sage}}

\begin{document}

\title[Lattices with many Borcherds products]{Lattices with many Borcherds products}

\author{Jan Hendrik Bruinier}
\email[J.~H.~Bruinier]{bruinier@mathematik.tu-darmstadt.de}
\author{Stephan Ehlen}
\email[S.~Ehlen]{ehlen@mathematik.tu-darmstadt.de}
\address{Fachbereich Mathematik, Technische Universit\"{a}t Darmstadt, Schlossgartenstra\ss{}e 7, D--64289 Darmstadt, Germany}
\author{Eberhard Freitag}
\email[E.~Freitag]{freitag@mathi.uni-heidelberg.de}
\address{Universit\"{a}t Heidelberg, Mathematisches Institut, Im Neuenheimer Feld 288, D--69120 Heidelberg, Germany}

\subjclass[2010]{11F12, 11E20, 11--04, 14C22}

\thanks{The first and the second author are partially supported by DFG grant BR-2163/4-1.}

\date{\today}

\begin{abstract}
We prove that there are only finitely many isometry classes of even lattices $L$ of signature $(2,n)$
for which the space of cusp forms of weight $1+n/2$ for the Weil representation of the discriminant group
of $L$ is trivial. We compute the list of these lattices.
They have the property that every Heegner divisor for the orthogonal group of $L$ can be realized as
the divisor of a Borcherds product.
We obtain similar classification results in greater generality for finite quadratic modules.
\end{abstract}

\maketitle
\section{Introduction}
\label{sect:intro}

Let $L$ be an even lattice of signature $(2,n)$ and write $\Orth(L)$ for its orthogonal group.
In his celebrated paper \cite{Bo1} R.~Borcherds constructed
a map from vector valued weakly holomorphic elliptic modular forms of weight $1-n/2$ to meromorphic modular forms for $\Orth(L)$
whose zeros and poles are supported on Heegner divisors. Since modular forms arising in this way have particular infinite product expansions,
they are often called {\em Borcherds products}. They play important roles in different areas such as
Algebraic and Arithmetic Geometry, Number Theory, Lie Theory, Combinatorics, and Mathematical Physics.

By Serre duality, the obstructions for the existence of weakly holomorphic modular forms with prescribed principal part
at the cusp at $\infty$ are given by vector valued cusp forms of dual weight $1+n/2$ transforming with the
Weil representation associated with the discriminant group of $L$  \cite{Bo2}.
In particular, if there are no non-trivial cusp forms of this type,
then there are no obstructions, and every Heegner divisor is the divisor of a Borcherds product.
A lattice with this property is called \emph{simple}.

It was conjectured by the third author that there exist only finitely
many isomorphism classes of such simple lattices. Under the
assumptions that $n\geq 3$ and that the Witt rank of $L$
(i.e.~the dimension of a maximal totally isotropic subspace of $L \otimes_\Z \Q$) is $2$,
it was proved by M.~Bundschuh
that there is an upper bound on the determinant of a simple lattice
\cite{Bu}. Unfortunately, this bound is very large and therefore not
feasible to obtain any classification results.
The argument of \cite{Bu} is based on volume estimates for Heegner
divisors and the singular weight bound for holomorphic modular forms
for  $\Orth(L)$.

The purpose of the present paper is twofold. First, we show that for
any $n \geq 1$ (without any additional assumption on the Witt rank)
there exist only finitely many isomorphism classes of even simple
lattices of signature $(2,n)$, see Theorem~\ref{fundest} and
Corollary~\ref{cor:fin}. Second, we develop an efficient algorithm to
determine all of these lattices. It turns out that there are exactly
$362$ isomorphism classes. Table~\ref{tab:survey} shows how many
of those occur in the different signatures. The corresponding
genus symbols (see Section \ref{sect:2}) of these lattices are listed in Tables \ref{tab:simple1},
\ref{tab:simple2}\footnote{Tables with global realizations can be obtained from \cite{CodeRepo}.}.

\begin{table}[h]
  \centering
    \caption{\label{tab:survey}
    Number of simple lattices of signature $(2,n)$.}
  \begin{tabular}{l ccccccccccccccc}
  %llllllllllllll
    \toprule
    $n$    & $1$   & $2$  & $3$  & $4$ & $5$ & $6$ & $7$ & $8$ & $9$ & $10$ & $11\leq n\leq 17$ & $18$ & $19\leq n\leq 25$ & $26$ & $n> 26$ \\\midrule
    \# & $256$ & $67$ & $15$ & $5$ & $3$ & $4$ & $2$ & $3$ & $3$ & $2$  & $0$               & $1$ & $0$                & $1$  & $0$\\\bottomrule
  \end{tabular}\vspace{3mm}
\end{table}

It is also interesting to record the Witt ranks of these lattices.
For $n=1$, we have $26$ anisotropic lattices.
The corresponding modular varieties are Shimura curves
while the remaining $230$ modular varieties for $n=1$ are modular curves.
For $n=2$, there are $24$ of Witt rank $1$ and $43$ of Witt rank $2$
but no anisotropic lattices. Finally, if $n \geq 3$, all simple lattices have Witt rank $2$.

Along the way, we obtain several results on modular forms associated
with finite quadratic modules which are of independent interest and which we
now briefly describe.  A finite quadratic module is a pair consisting of a
finite abelian group $A$ together with a $\Q/\Z$-valued non-degenerate
quadratic form $Q$ on $A$, see \cite{Ni}, \cite{Sk2}. Important examples of
finite quadratic modules are obtained from lattices.  If $L$ is an
even lattice with dual lattice $L'$, then the quadratic form on $L$
induces a $\Q/\Z$-valued quadratic form on the discriminant group $L'/L$.

Recall that there is a Weil representation $\rho_A$ of the the
metaplectic extension $\Mp_2(\Z)$ of $\Sl_2(\Z)$ on the group ring
$\C[A]$ of a finite quadratic module $A$, see Section \ref{sect:3}.
If $k\in \frac{1}{2}\Z$, we write
$S_{k,A}$ for the space of cusp forms of weight $k$ and representation
$\rho_A$ for the group $\Mp_2(\Z)$.
We say that a finite quadratic module $A$ is $k$-simple if $S_{k,A} = \{0\}$.
With this terminology, an even lattice $L$ is simple if and only if $L'/L$
is $(1+n/2)$-simple.

The dimension of the space $S_{k,A}$ can be computed by means of the
Riemann-Roch theorem. Therefore a straightforward approach to showing
that there are nontrivial cusp forms consists in finding lower bounds
for the dimension of $S_{k,A}$.
Unfortunately, the dimension formula \eqref{dim1} involves rather
complicated invariants of $\rho_A$ at elliptic and parabolic elements,
and it is a non-trivial task to obtain sufficiently strong bounds.
In the present paper we resolve this problem.
For instance, we obtain the following result
(see Theorem~\ref{fundest} and Corollary~\ref{cor:asy}).

\begin{theorem*}
For every $\eps>0$, there is a $C_{\eps} > 0$, such that
\[
  \left\lvert \dim(S_{k,A}) -\dim(M_{2-k,A(-1)}) -|A/\{\pm 1\}|\cdot\frac{k-1}{12} \right\rvert \leq C_\eps\cdot |A/\{\pm 1\}|\cdot N_A^{\eps-\frac{1}{2}}
\]
for every finite quadratic module $A$ and every weight $k\geq \tfrac{3}{2}$ with $2k\equiv -\sig(A)\pmod{4}$.
Here $N_A$ is the level of $A$, and $A(-1)$ denotes
the abelian group $A$ equipped with the quadratic form $-Q$.
\end{theorem*}

In Corollary~\ref{cor:fin}
we conclude that there exist only finitely many isomorphism classes of
finite quadratic modules $A$ with bounded number of generators such that
$S_{k,A}=\{0\}$ for some weight satisfying the condition of the theorem.
In particular, there are only finitely many isomorphism classes of simple lattices.
Note that there do exist infinitely many isomorphism classes of $\tfrac{1}{2}$-simple finite quadratic modules,
see Remark \ref{rem:finex}.

Since the constant implied in the Landau symbol in the above theorem is large,
it is a difficult task to compute the list of all $k$-simple finite quadratic modules
for a bounded number of generators.
We develop an efficient algorithm to address this problem.
The idea is to first compute all {\em anisotropic} finite quadratic modules
that are $k$-simple for some $k$.
To this end we derive an explicit formula for
$\dim(S_{k,A})$ in terms of class numbers of imaginary quadratic
fields and dimension bounds that are
 strong enough to obtain a classification (Theorem~\ref{a5'aniso} and
Table~\ref{simpleanisotable}).

Next we employ the fact that an arbitrary finite quadratic module $A$ has a
unique anisotropic quotient $A_0$, and that there are intertwining
operators for the corresponding Weil representations. For the
difference $\dim S_{k,A} - \dim S_{k,A_0}$ very efficient bounds can be
obtained. This can be used to classify all $k$-simple finite quadratic modules
with a bounded number of generators, see Algorithm \ref{algo} and the tables in Section \ref{sect:5}.

To resolve the problem of finding all simple {\em lattices} of
signature $(2,n)$, it remains to test which of these simple discriminant
forms arise as discriminant groups $L'/L$ of even lattices $L$ of
signature $(2,n)$. This is done in Section \ref{sec:simple-lattices} by
applying a criterion of \cite{CS}. Finally, in Section \ref{sect:bps}
we explain some applications of our results in the context of Borcherds products.

We thank R. Schulze-Pillot and N.-P.~Skoruppa for their help.
We also thank D.~Allcock and A.~Mark for their help with an earlier version of this paper.
Some of the computations were performed on computers at the University of Washington, USA,
supported by W.~Stein and the National Science Foundation Grant No. DMS-0821725.
We would also like to thank the anonymous referee for many helpful comments.

\section{Finite quadratic modules}
\label{sect:2}

Let $(A,Q)$ be a finite quadratic module (also called a finite quadratic form or discriminant form in the literature), that is,
a pair consisting of a finite abelian group $A$ together with a $\Q/\Z$-valued non-degenerate quadratic form $Q$ on $A$.
We denote the bilinear form corresponding to $Q$ by $(x,y) = Q(x+y) - Q(x) - Q(y)$.
Recall that $Q$ is called degenerate if there exists an $x \in A\setminus\{0\}$, such that $(x,y) = 0$ for all $y \in A$.
Otherwise, $Q$ is called non-degenerate.

The morphisms in the category of finite quadratic modules are group homomorphisms that
preserve the quadratic forms. In particular, two finite quadratic modules $(A,Q_A)$ and $(B,Q_B)$
are isomorphic if and only if there is an isomorphism of groups $\varphi: A \rightarrow B$,
such that $Q_B \circ \varphi = Q_A$.

In this section we collect some important facts about finite quadratic modules, which are well
known among experts but not easily found in the literature.
We mainly follow Skoruppa \cite{Sk3}. Other good references include \cite{Ni}, \cite{No}, \cite{Sk2}, \cite{Str}.

If $L$ is an even lattice, the quadratic form $Q$ on $L$
induces a $\Q/\Z$-valued quadratic form on the discriminant group $L'/L$ of $L$.
The pair $(L'/L,Q)$ defines a finite quadratic module, which we call the \emph{discriminant module} of $L$.
According to \cite{Ni}, any finite quadratic module can be obtained as the discriminant module of an even lattice $L$.
If $(b^+, b^-)$ then denotes the real signature of $L$,
the difference $b^+ - b^-$ is determined by its discriminant module $A$ modulo $8$ by Milgram's formula
\[
  \frac{1}{\sqrt{\abs{A}}} \sum_{a \in A} e(Q(a)) = e((b^{+} - b^{-})/8).
\]
Here and throughout we abbreviate $e(z) = e^{2\pi i z}$ for $z\in \C$.
We call $\sig(A) := b^{+} - b^{-}  \in \Z/8\Z$ the \emph{signature} of $A$.

We let $N$ be the level of $A$ defined by
\[
  N = \min\{ n\in \Z_{>0}\mid\; \text{$nQ(x)\in \Z$ for all $x\in A$}\}.
\]
It is easily seen that $N$ is a divisor of $2|A|$.

If $(A,Q_A)$ and $(B,Q_B)$ are two finite quadratic modules then the \emph{orthogonal direct sum}
$({A \oplus B}, Q_A + Q_B)$ also defines a finite quadratic module. Here $(Q_A+Q_B)(a+b) = Q_A(a) + Q_A(b)$ for $a \in A$ and $b \in B$.
We call a finite quadratic module \emph{indecomposable} if it is
not isomorphic to such a direct sum with non-zero $A$ and $B$.

The finite quadratic module $A$ is isomorphic to the orthogonal sum
of its $p$-components $A_p = A \otimes_\Z \Z_p$ with $p$ running through the primes.

Next we describe a list of indecomposable finite quadratic modules.

\begin{definition} \label{elemdf}
  Let $p$ be a prime and $t$ be an integer not divisible by $p$.
  We define the following elementary finite quadratic modules.
  \begin{align*}
     \calA_{p^{k}}^{t} &= \left( \Z/p^{k}\Z, \frac{tx^{2}}{p^{k}} \right) \text{ for } p > 2, \\ %\label{eq:fqms:Ap}\\
     \calA_{2^{k}}^{t} &= \left( \Z/2^{k}\Z, \frac{tx^{2}}{2^{k+1}} \right), \\%\label{eq:fqms:A2}\\
     \calB_{2^{k}} &= \left( \Z/2^{k}\Z \oplus \Z/2^{k}\Z, \frac{x^{2}+2xy+y^{2}}{2^{k}} \right), \\ %\label{eq:fqms:B2}\\
     \calC_{2^{k}} &= \left( \Z/2^{k}\Z \oplus \Z/2^{k}\Z, \frac{xy}{2^{k}} \right), %\label{eq:fqms:C2}
  \end{align*}
\end{definition}

\begin{theorem}
  \label{thm:jordan}\quad
  \begin{enumerate}
  \item The finite quadratic modules listed in Definition \ref{elemdf} are indecomposable.
  \item Every indecomposable finite quadratic module is isomorphic to a finite quadratic module as in Definition \ref{elemdf}.
  \item Moreover, every finite quadratic module is isomorphic to a direct sum of indecomposable finite quadratic modules.
  \end{enumerate}
\end{theorem}
\begin{proof}
  Statement (1) is clear. The other two statements follow from
  the classification of $p$-adic lattices. See \cite{Ni} for details, in particular Proposition 1.8.1.
\end{proof}

Consider a decomposition of the $p$-components of a finite quadratic module $A$ as a direct sum
\[
  A_p = A_{p,1} \oplus \ldots \oplus A_{p,l_p},
\]
where each $A_{p,i}$ is a direct sum of elementary finite quadratic modules
$\calA_{q_i}^{t_i}, \calB_{q_i}, \calC_{q_i}$, with $q_{i} = p^{r_i}$, $r_i \geq 1$ and $q_1 < \ldots < q_{r}$.

Such a decomposition is called a \emph{Jordan decomposition}
of $A$ and the direct summands $A_{p,i}$ are called \emph{Jordan components} of $A$.
We will also call any finite quadratic module that is isomorphic to a direct sum
of elementary finite quadratic modules $\calA_{q}^{t}, \calB_{q}, \calC_{q}$ for a fixed $q$
a Jordan component.

We will now describe a handy notation for such a Jordan decomposition.
The symbols we use are essentially those introduced by Conway and Sloane \cite{CS} for the genus of an integral lattice,
that is, its class under rational equivalence.
The following statement (see \cite{Ni}, Corollary 1.16.2) motivates the use of their symbols for us.

\begin{proposition}
  Two even lattices $L$ and $M$ that have the same real signatures
  have isomorphic finite quadratic modules if and only if $L$ and $M$ are in the same genus.
\end{proposition}

The following two lemmas are straightforward to prove (see also \cite{Ni}, Proposition 1.8.2).
\begin{lemma}
  \label{lem:isomodd}
  Let $p>2$ be a prime and let $q=p^{r}$ for a positive integer $r$.
  \begin{enumerate}
  \item We have $\calA_q^t \cong \calA_q^s$ if and only if $\leg{s}{p} = \leg{t}{p}$.
  \item Suppose that $\leg{2s}{p}=1$ and $\leg{2t}{p}=-1$.
            Then $\calA_q^{s} \oplus \calA_q^{s} \cong \calA_{q}^{t} \oplus \calA_q^{t}$.
  \item
    In particular, if $A$ is a Jordan component of the form
    \[
      A = \bigoplus_{i=1}^{n} \calA_q^{t_{i}},
    \]
    then $A$ is isomorphic to
    \[
      \calA_{q}^{t} \oplus \bigoplus_{i=1}^{n-1} \calA_{q}^{s}, \text{ with } \prod_{i=1}^{n} \leg{2t_{i}}{p} = \leg{2t}{p}
    \]
    for any $s$ with $\leg{2s}{p}=1$.
  \end{enumerate}
\end{lemma}

\begin{lemma}
  \label{lem:isomeven}
  Let $q = 2^{r}$ for a positive integer $r$.
  Moreover, let $s, t$ be odd integers.
  \begin{enumerate}
  \item We have $\calA_2^{s} \equiv \calA_{2}^{t}$ if and only if $s \equiv t \pmod{4}$.
  \item If $r > 1$, then $\calA_{q}^{s} \equiv \calA_{q}^{t}$ if and only if $s \equiv t \pmod{8}$.
  \item Let $s_1,\ldots,s_{n},t_1, \ldots, t_n \in \Z$ such that
        \[
          \sum_{i=1}^{n} s_i \equiv  \sum_{i=1}^{n} t_{i} \pmod{8} \text{ and }
          \prod_{i=1}^{n} \leg{s_i}{2} = \prod_{i=1}^{n}{t_{i}}.
        \]
        Then
        \[
          \bigoplus_{i=1}^{n} \calA_{q}^{s_{i}} \cong \bigoplus_{i=1}^{n} \calA_{q}^{t_{i}}.
        \]
  \item We have $\calB_q \oplus \calB_q \cong \calC_q \oplus \calC_q$.
  \item Moreover, we have $\calA_{q}^{t} \oplus \calB_q \cong \calA_q^{t_1} \oplus \calA_{q}^{t_{2}} \oplus \calA_{q}^{t_{3}}$
        with $t_1 + t_2 + t_3 \equiv t \pmod{8}$ and
        \[
          \prod_{i=1}^{3}\leg{t_i}{2} = - \leg{t}{2}.
        \]
  \item Finally, $\calA_{q}^{t} \oplus \calC_q \cong \calA_q^{t_1} \oplus \calA_{q}^{t_{2}} \oplus \calA_{q}^{t_{3}}$
        with $t_1 + t_2 + t_3 \equiv t \pmod{8}$ and
        \[
          \prod_{i=1}^{3}\leg{t_i}{2} = \leg{t}{2}.
        \]
  \end{enumerate}
\end{lemma}

\begin{definition}
  \label{def:genus-symbols}
  Using the preceding lemmas, we define a symbol for a Jordan decomposition of a finite quadratic module as follows.
  First of all, by convention, we write $1^{+1}$ or $1^{-1}$ for the trivial module $A=\{0\}$ with
  the $0$-map as quadratic form.

  Now let $A$ be a Jordan component, $p$ be a prime and $q = p^{r}$.
  \begin{enumerate}
  \item If $p$ is odd, the two isomorphism classes of Jordan components in Lemma \ref{lem:isomodd}, (3)
        are denoted $q^{\pm n}$, where $\leg{2t}{p} = \pm 1$.
  \item Let $p=2$.
    \begin{enumerate}
    \item We write $q^{\pm n}_{t}$ if $A$ is isomorphic to $\calA_{q}^{t_{1}} \oplus \ldots \oplus \calA_{q}^{t_{n}}$
          with $t_1 +\ldots + t_n  \equiv t \pmod{8}$ and $\leg{t_1}{2} \cdots \leg{t_n}{2} = \pm 1$.
          We normalize $t$ to be contained in the set $\{1,3,5,7\}$ and if $q = 2$, we take $t \in \{1,7\}$.
    \item We write $q^{+2n}$ if $A$ is isomorphic to $n$ copies of $\calC_{q}$.
    \item And we write $q^{-2n}$ if $A$ is isomorphic to $n-1$ copies of $\calC_q$ and one copy of $\calB_q$.
    \end{enumerate}
  \end{enumerate}
  For a general finite quadratic module, we concatenate the symbols of the Jordan components as defined above.
\end{definition}

\begin{example}
  The Jordan decomposition $\calA_{2}^{1} \oplus (\calA_{3}^{1} \oplus \calA_{3}^{1})$ has the symbol $2_{1}^{+1}3^{+2}$
  and $2^{-2}4_{3}^{+3}3^{-1}$ is the symbol for the Jordan decomposition
  $\calB_{2} \oplus (\calA_{4}^{1} \oplus \calA_{4}^{1} \oplus \calA_{4}^{1}) \oplus \calA_{3}^{1}$.
\end{example}

\begin{proposition}
  \label{prop:genus-symbol-isom}
   Let $A$ and $B$ be finite quadratic modules and let $p > 2$ be a prime.
   If $A_p \cong B_p$, then the corresponding $p$-components of the genus symbols of $A$
   and $B$ coincide.
\end{proposition}
\begin{proof}
  This follows from the uniqueness of the Jordan decomposition in the case of an odd prime $p$
  (see, for instance, Theorem 5.3.2 in \cite{Ki}).
\end{proof}

\begin{remark}
  In contrast to Proposition \ref{prop:genus-symbol-isom}, note that the symbol (and the Jordan decomposition)
  for $p = 2$ is not uniquely determined by the isomorphism class.
  For instance, the Jordan decompositions $2_{1}^{+1}4_{1}^{+1}$ and $2_{7}^{+1}4_{3}^{-1}$
  correspond to isomorphic finite quadratic modules. See also Theorem \ref{thm:disc-iso} in the next section.
\end{remark}

For an integer $n$ we define the $n$-torsion subgroup of a finite quadratic module $A$ by
\[
  A[n]=\left\{\gamma \in A\mid \; n\gamma = 0\right\}.
\]
Moreover, we let $A^n$ be the image of the multiplication by $n$ map. Then we have the exact sequence
\[
  0\longrightarrow A[n] \longrightarrow A \longrightarrow A^n\longrightarrow  0,
\]
and $A$ is the orthogonal sum of $A[n]$ and $A^n$.
It follows from the 
theorem of elementary divisors and the 
Jordan decompositon (Theorem \ref{thm:jordan})
that
\begin{equation}\label{lnest}
  |A[n]|\leq 2n\frac{|A|}{N}.
\end{equation}
The quantity $d=d_A:= |A/\{\pm 1\}|$ can be expressed in term of the $2$-torsion as
\begin{align}
  \label{eq:da}
  d=\frac{|A|}{2}+\frac{|A[2]|}{2}.
\end{align}

\subsection{Gauss sums and divisor sums}
\label{sec:gauss-sums-divisor}

We now collect some facts about Gauss sums and divisor sums
associated to finite quadratic modules which we will need later.
For an integer $n\in \Z$ the Gauss sum $G(n,A)$ is defined by
\begin{equation}\label{defgauss}
  G(n,A)=
  \sum_{\gamma\in A} e(n Q(\gamma)).
\end{equation}
We have the elementary properties
\begin{align}
  \label{basic1}
  G(-n,A)   &= \overline{G(n,A)}, \\
  \label{basic2}
  G(n+N,A) &= G(n,A), \\
  \label{basic3}
  G(n,A \oplus B) &= G(n,A)G(n,B).
\end{align}
The following lemma is a consequence of \cite[Lemma 3.1]{Bo3}.

\begin{lemma}\label{gauss}
If $(n,N)=1$ then $|G(n,A)|=\sqrt{|A|}$. For general $n$ we have the estimate
  \[
    |G(n,A)|\leq \sqrt{|A|}\sqrt{|A[n]|}.
  \]
\end{lemma}

We will also need explicit formulas for the Gauss sums in some cases in Section \ref{sec:anis-discr-forms}.
These are easily proven by relating $G(n,A)$ to the standard Gauss sums
(see, for instance \cite{Str}).
\begin{proposition}
\label{prop:gaussodd}
Let $p > 2$ be a prime. We have $G(n,p^{\pm 1}) = p$ if $p \mid n$ and
\[
  G(n,p^{\pm 1}) = \pm \sqrt{p} \leg{p}{2} \leg{n}{p} e\left(\frac{1-p}{8}\right)
\]
if $(n,p) = 1$.
\end{proposition}

\begin{proposition}
  \label{prop:gausseven1}
  Let $q = 2^r$. For $n \in \Z$ we put $n' = n/(n,q)$ and $q' = q/(n,q)$. We have
  \[
    G(n,q_{t}^{\pm 1}) = \sqrt{q} \sqrt{(n,q)} \leg{t n'}{q'} \cdot
        \begin{cases}
            e\left( \frac{tn'}{8}\right), & \text{if } q \nmid n,\\
            0, & \text{if } q \mid\mid n,\\
            1, & \text{if } 2q \mid n.
        \end{cases}
  \]
\end{proposition}

\begin{proposition}
  \label{prop:gausseven2}
  Let $q = 2^r$ with $r \geq 1$.
  We have
  \[
    G(n,q^{-2}) = q (q,n) \leg{3}{q'}
  \]
  and
  \[
    G(n,q^{+2}) = q (q,n).
  \]
\end{proposition}

The following theorem is used later on to decide when two given genus symbols
correspond to isomorphic finite quadratic modules.

\begin{theorem}[\cite{Sk3}] \label{thm:disc-iso}
  Let $A$ and $B$ be finite quadratic modules.
  Then $A$ and $B$ are isomorphic if and only if their underlying
  abelian groups are isomorphic and
  \[
       G(n,A) = G(n,B)
  \]
  for all divisors $n$ of the level of $A$ and $B$.
\end{theorem}
\begin{proof}
  It is clear that the condition $G(n,A) = G(n,B)$ is necessary.
  Using this, it is easy to prove the theorem for $p$-components with $p > 2$.
  If $A_p$ and $B_p$ are finite quadratic modules of prime power level $p^{r}$ with $G(p^{i},A_p) = G(p^{i},B_p)$
  for all $i \in {0, \ldots, r}$, then $A_p \cong B_p$ follows from Lemma \ref{lem:isomodd} and the explicit formula
  for the Gauss sum in Proposition \ref{prop:gaussodd}.
  The general case is treated in \cite{Sk3} in detail.
\end{proof}

For $s\in \C$ we define a divisor sum $\sigma(s,A)$ associated to $A$ by
\begin{align}
  \label{eq:divsum}
  \sigma(s,A) = \sum_{a\mid N} a^s \sqrt{|A[a]|}.
\end{align}
Here the sum runs over all positive divisors $a$ of the level $N$ of $A$.
If $B$ is another finite quadratic module of level $N'$ coprime to $N$, then
\[
\sigma(s,A\oplus B)= \sigma(s,A)\sigma(s,B).
\]
Consequently, $\sigma(s,A)$ is the product of the $\sigma(s,A_p)$ for $p$ running through the primes dividing $N$.

\begin{lemma}\label{lem:sigest}
For $s\in \R$ we have the estimate
\[
\sigma(s,A)\leq \sqrt{\frac{2|A|}{N}} \cdot\sigma_{s+1/2}(N),
\]

where $\sigma_s(N)= \sum_{a\mid N} a^s$ denotes the usual divisor sum.
\end{lemma}

\begin{proof}
This is a direct consequence of the estimate \eqref{lnest}.
\end{proof}

\section{Vector valued modular forms}
\label{sect:3}

We write $\Mp_2(\Z)$ \label{bi1} for the metaplectic extension of $\Sl_2(\Z)$,
realized as the group of pairs
$(M,\phi(\tau))$, where $M=\kabcd\in\Sl_2(\Z)$ and $\phi$
is a holomorphic function on the upper complex half plane $\H$ with $\phi(\tau)^2=c\tau+d$ (see e.~g.~\cite{Bo1}, \cite{Br1}).
It is well known that $\Mp_2(\Z)$ is generated by
\[
T= \left( \zxz{1}{1}{0}{1}, 1\right)\qquad\text{and}\qquad S= \left(
  \zxz{0}{-1}{1}{0}, \sqrt{\tau}\right).
\]
One has the relations $S^2=(ST)^3=Z$, where $Z=\left(
  \kzxz{-1}{0}{0}{-1}, i\right)$ is the standard generator of the
center of $\Mp_2(\Z)$.

The Weil representation associated with $A$ is a unitary representation $\rho_A$ of
$\Mp_2(\Z)$ on the group
algebra $\C[A]$. If we denote the standard basis of
$\C[A]$ by $(\frake_\gamma)_{\gamma\in A}$ then $\rho_A$ can
be defined by the action of the generators $S,T\in\Mp_2(\Z)$ as
follows (see also \cite{Sk2}, \cite{Bo1}, \cite{Br1}, where the dual of $\rho_A$
is used):
\begin{align}\label{actionT}
  \rho_A(T)\frake_\gamma &= e(-Q(\gamma)) \frake_\gamma,\\
  \label{actionS}
  \rho_A(S)\frake_\gamma &= \frac{ e(\sig(A)/8)}{\sqrt{|A|}}
  \sum_{\delta\in A} e((\gamma,\delta)) \frake_\delta.
\end{align}
% This representation is essentially the Weil representation attached to
% the quadratic module $(\calL,q)$ (see \cite{No}).

Let $k\in \tfrac{1}{2}\Z$.  We denote by $M_{k,A}$ the vector space of
$\C[A]$-valued modular forms of weight $k$ with representation
$\rho_A$ for the group $\Mp_2(\Z)$. The subspace of cusp forms is
denoted by $S_{k,A}$.  It is
easily seen that $M_{k,A}=0$, if $2k\not\equiv \sig(A) \pmod{2}$.

The dimension of the vector space $M_{k,A}$ can be computed using the
Riemann-Roch theorem or the Selberg trace formula.  This is carried
out in \cite{Fr} and \cite{Fi} in a more general situation. In our
special case the following formula holds (see \cite{Bo2} p.~228 and
\cite{Fr} Chapter~8.5, Theorem~5.1). For simplicity we assume that
$2k\equiv
-\sig(A)\pmod{4}$, since our application to simple lattices will only concern this case.
Then the
$d$-dimensional subspace
$W=\Span\{\frake_\gamma+\frake_{-\gamma};\; \gamma\in A\}$ of
$\C[A]$ is preserved by $\rho_A$, and $\rho_A(Z)$ acts by
multiplication with $e(-k/2)$ on $W$.  We denote by $\rho$ the
restriction of $\rho_A$ to $W$.  If $M$ is a unitary matrix of size
$d$ with eigenvalues $e(\nu_j)$ and $0\leq \nu_j<1$ (for
$j=1,\dots,d$), we define
\[
\alpha(M)=\sum_{j=1}^d \nu_j.
\]
If $k\geq \tfrac{3}{2}$, the dimension of $M_{k,A}$ is given by
\begin{align}\label{dim1}
  \dim( M_{k,A} )& = d+dk/12-\alpha\left(e^{\pi i k/2}\rho(S)\right) - \alpha\left(\left( e^{\pi i k/3}\rho(ST)\right)^{-1}\right) -\alpha(\rho(T))\\
  \nonumber
  &\phantom{=} {}+\dim( S_{2-k,A(-1)} ).
\end{align}
Furthermore, the dimension of $S_{k,A}$ is given by
\begin{align}\label{eq:dimcusp}
  \dim( S_{k,A} ) =  \text{first line of } \eqref{dim1} - \left|\left\{\gamma\in A /\{\pm 1\}; \; Q(\gamma)\in  \Z \right\}\right|
+\dim( M_{2-k,A(-1)} ).
\end{align}
%(see also \cite{Br1}, Chapter 1.2.3).
Here $A(-1)$ denotes the finite quadratic module given by the abelian group $A$ equipped with the quadratic form $-Q$.
If $k>2$, then $M_{2-k,A(-1)}$ vanishes. If $k=2$, then
$M_{0,A(-1)}$ is equal to the space of $\Mp_2(\Z)$-invariants in $\C[A]$ for the dual representation
of $\rho_A$. Finally, when $k=\tfrac{3}{2}$, according to the Serre-Stark theorem, the space $M_{\frac{1}{2},A(-1)}$ is generated by unary theta series. It was explicitly computed by Skoruppa in \cite{Sk1} and \cite{Sk2} as follows. For every non-zero $l\in\Z$ we write $V(l)$ for the finite quadratic module of level $4|l|$ given by $\Z/2l\Z$ equipped with the quadratic form $Q(x)=\frac{1}{4l}x^2$. Let $\epsilon$ be the automorphism of $V(l)$ given by multiplication by $-1$, and write $\C[V(l)]^\epsilon$ for the corresponding space of invariants.
According to \cite[Theorem 8]{Sk2} we have
\begin{align}
\label{eq:dimwgt1/2}
M_{\frac{1}{2},A(-1)} \cong\bigoplus_{\substack{l>0,\; 4l\mid N\\ \text{$N/4l$ squarefree}}} \big(\C[V(-l)]^\epsilon\otimes \C[A(-1)]\big)^{\Mp_2(\Z)}.
\end{align}
Here the action of $\Mp_2(\Z)$ on the tensor products on the right hand side is given by the Weil representation.

\section{Dimension estimates}

\label{sect:de}

In this section we derive lower bounds
for the dimension of $S_{k,A}$. In view of (\ref{dim1}) and
(\ref{eq:dimcusp}) we have to estimate the quantities
\begin{align*}
  \alpha_1 &:= \alpha\left(e^{\pi i k/2}\rho(S)\right),\\
  \alpha_2 &:= \alpha\left(\left( e^{\pi i k/3}\rho(ST)\right)^{-1}\right),\\
  \alpha_3 &:= \alpha(\rho(T)), \\
  \alpha_4 &:= \big|\left\{\gamma\in A/\{\pm 1\}; \;
    Q(\gamma)\in \Z \right\}\big|.
\end{align*}
We begin by recalling some trivial bounds from \cite[Bemerkungen 2.2.1 and 2.2.5]{Bu}.
We have
\begin{align*}
  \alpha_1 \leq \frac{1}{2}d, \qquad
  \alpha_2 \leq \frac{2}{3}d, \qquad
  \alpha_3 +\alpha_4\leq d.
\end{align*}
If we insert these bounds into \eqref{eq:dimcusp} we obtain the following corollary.

\begin{corollary}\label{cor:kbound}
If $k> 14$ and $2k\equiv -\sig(A)\pmod{4}$, then $S_{k,A}\neq \{0\}$.
\end{corollary}
Note that this bound on $k$ is sharp, since there are no nontrivial scalar valued cusp forms of weight $14$ for $\SL_2(\Z)$.

To prove the existence of non-trivial cusp forms for smaller values of $k$ by means of the dimension
formula, we need much better estimates for the $\alpha_i$.
The quantities $\alpha_1$ and $\alpha_2$ can be expressed in terms of Gauss sums
associated with $A$. By means of the estimate in Lemma \ref{gauss},
we obtain the following result (see Lemma 2 and Corollary 3 in \cite{Br:Pic}).

\begin{lemma}
  \label{a1a2est}
  The quantities $\alpha_1$ and $\alpha_2$ satisfy the estimates
  \begin{align}
    \label{a1est}
    |\alpha_1 -d/4 | & \leq \frac{1}{4}\sqrt{|A[2]|},\\
    \label{a2est}
    |\alpha_2 -d/3 | & \leq \frac{1}{3\sqrt{3}} \left( 1 +
      \sqrt{|A[3]|} \right).
  \end{align}
\end{lemma}

%We now derive an estimate for $\alpha_4$.
%If $n$ is a positive integer, we define the divisor sum $\sigma_t(n)=\sum_{a| n} a^t$.

\begin{lemma}\label{a4}
  We have
  \[
  |\alpha_4|\leq \frac{|A[2]|}{2} +\frac{\sqrt{|A|}}{2} \sigma(-1,A),
  \]
  where $\sigma(-1,A)$ is the divisor sum defined in \eqref{eq:divsum}.
\end{lemma}

\begin{proof}
  We write $\alpha_4$ as
  \[
  \alpha_4 = \frac{1}{2}\sum_{\substack{\gamma\in A[2]\\
      Q(\gamma)\in\Z}} 1 +\frac{1}{2}\sum_{\substack{\gamma\in A\\
      Q(\gamma)\in\Z}} 1.
  \]
  The second term on the right hand side is equal to
  \[
  \frac{1}{2N}\sum_{\gamma\in A}\sum_{\nu\,(N)} e(Q(\gamma)\nu)
  =\frac{1}{2N} \sum_{\nu\,(N)} G(\nu,A).
  \]
  Using Lemma \ref{gauss}, we obtain
  \begin{align*}
    |\alpha_4|&\leq \frac{|A[2]|}{2} + \frac{1}{2N} \sum_{\nu\,(N)} \sqrt{|A|}\sqrt{|A[\nu]|}\\
    &\leq \frac{|A[2]|}{2} + \frac{\sqrt{|A|}}{2N} \sum_{a\mid N}\sum_{\substack{\mu\,(N/a)\\ (\mu,N/a)=1}}\sqrt{|A[a\mu]|}\\
    &\leq \frac{|A[2]|}{2} + \frac{\sqrt{|A|}}{2N} \sum_{a|N} \frac{N}{a} \sqrt{|A[a]|}\\
    &\leq \frac{|A[2]|}{2} + \frac{\sqrt{|A|}}{2} \sigma(-1,A).
    % \sum_{a|N}\frac{1}{a} \sqrt{|A[a]|}\\
    % &\leq \frac{|\calL^2|}{2} + \frac{\sqrt{|\calL|}}{2D} \sum_{a|D} \frac{D}{a} a^{r/2}\\
    % &\leq \frac{|\calL^2|}{2} + \frac{\sqrt{|\calL|}}{2}
    % \sigma_{r/2-1}(D).
  \end{align*}
  This concludes the proof of the lemma.
\end{proof}

%\bigskip

Before we consider $\alpha_3$, we introduce some additional notation. If
$x\in \R$, we write $[x]=\max\{n\in\Z;\; n\leq x\}$ for the greatest-integer function. Moreover, we let
\begin{equation}\label{b1}
  \B(x)=x-\tfrac{1}{2}([x]-[-x]).
\end{equation}
be the $1$-periodic function on $\R$ with $\B(x)=0$ for
$x=0,1$ and $\B(x)=x-1/2$ for $0<x<1$.  By definition
\[
\alpha_3 = \sum_{\gamma\in A/\{\pm 1\}} \left(
  -Q(\gamma)-[-Q(\gamma)]\right).
\]
Using $\B(x)$ and $\alpha_4$ we may rewrite this in the form
\begin{align*}
\alpha_3 &= \frac{d}{2}-\frac{\alpha_4}{2} -\sum_{\gamma\in A/\{\pm
  1\}}\B(Q(\gamma))\\
  &= \frac{d}{2}-\frac{\alpha_4}{2}-\frac{1}{2}\sum_{\gamma\in A[2]}\B(Q(\gamma)) -
\frac{\beta}{2},
\end{align*}
where
%Hence, to obtain information on $\alpha_3$, it suffices to consider
%the invariants
\begin{align}
%\label{a5}
%  \alpha_5&=\sum_{\gamma\in A/\{\pm 1\}}\B(Q(\gamma)),\\
  \label{a5'}
  \beta &= \sum_{\gamma\in A}\B(Q(\gamma)).
\end{align}
%of $L$.
%We have the relation
%\[
%\alpha_5 = \frac{1}{2}\sum_{\gamma\in A[2]}\B(Q(\gamma)) +
%\frac{\alpha_5'}{2}.
%\]
For $\gamma\in A[2]$ we have $Q(\gamma)\in \frac{1}{4}\Z$,
and therefore $|\B(Q(\gamma))|\leq 1/4$. Hence
\begin{align}
  % |\alpha_5| &\leq |A[2]|/8 + |\beta|/2,\\
  \label{a35'}
  |\alpha_3 - d/2+\alpha_4/2| & \leq |A[2]|/8 + |\beta|/2.
\end{align}

% The main result of this section is the following estimate for $\beta$.

\begin{lemma}\label{a5'est}
  The quantity $\beta$ satisfies
  \[
  |\beta|\leq \frac{\sqrt{|A|}}{\pi}\left( \frac{3}{2} +\ln(N)\right)
  \left(\sigma(-1,A)-\frac{\sqrt{|A|}}{N}\right).
  \]
\end{lemma}

\begin{proof}
Exactly as in the proof of \cite[Lemma 5]{Br:Pic}, we derive
  \begin{align*}
    | \beta|
    &\leq \frac{\sqrt{|A|}}{\pi }
    \sum_{\nu=1}^{N-1}\frac{1}{\nu}\sqrt{|A[\nu]|}
    +\frac{\sqrt{|A|}}{2\pi N} \sum_{\nu=1}^{N-1}
    \sqrt{|A[\nu]|} .
  \end{align*}
%  The latter sum over $n$ equals $1$.
Rewriting the sum
%If we apply (\ref{lnest}) and rewrite
over $\nu$, we obtain
  \begin{align*}
    | \beta| &\leq \frac{\sqrt{|A|}}{\pi } \sum_{\substack{a|N\\
        a\neq N}}\sum_{\substack{\mu=1\\(\mu,N/a)=1}}^{N/a}
    \frac{1}{a\mu} \sqrt{|A[a]|}
    +\frac{\sqrt{|A|}}{2\pi N} \sum_{\substack{a|N\\ a\neq N}} \sum_{\substack{\mu=1\\(\mu,N/a)=1}}^{N/a} \sqrt{|A[a]|} \\
    &\leq   \frac{\sqrt{|A|}}{\pi } \sum_{\substack{a|N\\ a\neq N}}(1+\ln(N/a))\frac{1}{a} \sqrt{|A[a]|} +\frac{\sqrt{|A|}}{2\pi N} \sum_{\substack{a|N\\ a\neq N}} \frac{N}{a}\sqrt{|A[a]|} \\
    &\leq \frac{\sqrt{|A|}}{\pi}\left( \frac{3}{2} +\ln(N)\right)
    \left(\sigma(-1,A)-\frac{\sqrt{|A|}}{N}\right).
  \end{align*}
  Here we have also used the estimate $\sum_{\nu=1}^n \frac{1}{\nu}\leq
  1+\ln(n)$.
\end{proof}

Putting the above lemmas together, we obtain the following
estimate for the dimension of the space $S_{k,A}$.

\begin{theorem}\label{fundest}
%Assume that $2k\equiv b^--b^+\pmod{4}$.
If $k\geq \tfrac{3}{2}$ and $2k\equiv -\sig(A)\pmod{4}$, then
  \begin{align*}
    \left|\dim(S_{k,A})-\dim(M_{2-k,A(-1)})-\frac{d(k-1)}{12}\right|\leq R(A),
  \end{align*}
  where
  \begin{align*}
    R(A)&=\frac{\sqrt{|A[2]|}}{4} +\frac{1+\sqrt{|A[3]|}}{3\sqrt{3}} + \frac{3}{8}|A[2]| \\
    &\phantom{=}{}+ \frac{\sqrt{|A|}}{4}\sigma(-1,A) + \frac{\sqrt{|A|}}{2\pi}\left( \frac{3}{2}
      +\ln(N)\right) \left(\sigma(-1,A)-\frac{\sqrt{|A|}}{N}\right)
  \end{align*}
  is independent of $k$.
\end{theorem}

\begin{proof}
The dimension formula \eqref{eq:dimcusp} states that
  \begin{align*}
%    \label{eq:df2}
    \dim(S_{k,A})-\dim(M_{2-k,A(-1)})&= \frac{d(k+12)}{12}-\alpha_1-\alpha_2-\alpha_3-\alpha_4\\
    \nonumber
    &= \frac{d(k-1)}{12} -(\alpha_1-\frac{d}{4}) -(\alpha_2-\frac{d}{3}) -(\alpha_3-\frac{d}{2}+\frac{\alpha_4}{2})-  \frac{\alpha_4}{2}.
  \end{align*}
  Employing \eqref{a35'}, Lemma \ref{a1a2est}, Lemma \ref{a4}, and Lemma \ref{a5'est}, we obtain the assertion.
%
%If $k=2$ the same argument applies, noticing that $0\leq \dim (M_{0,A(-1)})\leq \alpha_4$. Hence we can bound  $|\dim (M_{k,A(-1)})-\frac{\alpha_4}{2}|$ by $\alpha_4/2$ and then apply Lemma \ref{a4}.
\end{proof}

\begin{corollary}
\label{cor:asy}
For every $\eps>0$ there exists a constant $C$ (independent of $k$ and $A$)
such that
\[
\left|\dim (S_{k,A}) -\dim(M_{2-k,A(-1)})-\frac{d(k-1)}{12}\right| \leq C d N^{\eps-\frac{1}{2}}.
\]
for every finite quadratic module $A$ and every weight
$k\geq \tfrac{3}{2}$ with $2k\equiv -\sig(A)\pmod{4}$.
\end{corollary}

\begin{proof}
%We first assume that $k\geq 2$.
Using Theorem \ref{fundest}, the bound \eqref{lnest} for $|A[a]|$,
and Lemma \ref{lem:sigest}, we find that there are constants $C_1, C_2>0$ (independent of $k$ and $A$)
such that
\[
R(A)\leq C_1\frac{|A|}{N}+ C_2\frac{|A|}{\sqrt{N}}\sigma_{-1/2}(N)(1+\ln(N)).
\]
By  means of the estimate $\sigma_{-1/2}(N)\ll_\eps N^\eps$ we see that there exists a $C>0$ (depending on $\eps$) such that
\[
R(A)\leq C \cdot d  N^{\eps-\frac{1}{2}}.
\]
This proves the assertion.
%\[
%\dim(S_{k,A})\geq d\left( \frac{k-1}{12}-C N^{\eps-1/2}\right).
%\]
\end{proof}

\begin{corollary}\label{cor:fin}
Let  $r_0\in \Z_{\geq 0}$. There exist only finitely many isomorphism classes of finite quadratic modules $A$ with minimal number of generators $\leq r_0$ such that $S_{k,A}=\{0\}$ for some weight $k \geq \tfrac{3}{2}$ with $2k\equiv -\sig(A)\pmod{4}$.
\end{corollary}
\begin{proof}
Since for any $N_0\in \Z_{>0}$ there are only finitely many isomorphism classes of
finite quadratic modules $A$ with bounded minimal number of generators and level $N\leq N_0$, we obtain the assertion from Corollary \ref{cor:asy}.
\end{proof}
\begin{remark}
\label{rem:finex}
i) In Corollary \ref{cor:fin}, the bound $r_0$ on the minimal number of generators is essential.
%If we allow $r$ to vary, then there are infinitely many isomorphism classes of finite quadratic modules $A$ such that $S_{k,A}=\{0\}$.
For instance, if $A=3^{\epsilon n}$ with $n\in \Z_{>0}$ odd and $\epsilon = (-1)^{\frac{n-1}{2}}$, then $\sig(A)\equiv 2\pmod{4}$ and
$S_{3,A} =\{0\}$. This follows for instance from the dimension formula in \cite{Ha}, Chapter~5.2.1, p.~93.

ii) Note that if $k=\tfrac{1}{2}$, it follows from \cite[Theorem 7]{Sk2} that there exist infinitely many isomorphism classes of finite quadratic modules $A$ such that $S_{\frac{1}{2},A}=\{0\}$. It would be interesting to understand what happens in weight $1$.
\end{remark}

Under the assumptions of Corollary \ref{cor:fin} it is possible to make the constants appearing in the proof explicit and to derive an explicit lower bound $N_0$ such that  $S_{k,A}$ is nontrivial for all finite quadratic modules $A$ with level larger than $N_0$. However, it turns out that such a bound is very large, and therefore not useful for a computer computation of the finite list of simple finite quadratic modules $A$.
As an example, for $\varepsilon = 1/5$ we can get $C = 45.38$ and this would give the bound 
$N _0 \geq 1.32 \cdot 10^{9}$ for $k=2$. 
Therefore, a search for finite quadratic modules with order up to $3.04 \cdot 10^{36}$ would be required in the case of signature $(2,2)$, which is not feasible.
Even though it might be possible to find better parameters, we did not try to optimize this.
Instead, we use a more systematic approach to this computational task.
We first compute all {\em anisotropic} simple finite quadratic modules, and then construct all remaining ones by
means of isotropic quotients.

% By collecting some of the terms in the proof of Theorem 4.5, we get}
% \begin{align*}
% R(A) &\leq \frac{\sqrt{\abs{A}}}{2\sqrt{N}}+\frac{1}{3\sqrt{3}}+\frac{2\sqrt{\abs{A}}}{\sqrt{3}\sqrt{N}}
% +\frac{3\abs{A}}{2N} + \frac{\abs{A}}{\sqrt{2}\sqrt{N}} \sigma_{-1/2}(N)(\frac{1}{\pi} + (1+\ln(N)))\\
% &\leq \frac{\abs{A}}{N}(\frac{1}{\sqrt{2}} +\frac{\sqrt{2}}{3\sqrt{3}}
% +\frac{2\sqrt{2}}{\sqrt{3}}+\frac{3}{\sqrt{2}})
% + 11.6\frac{\abs{A}}{\sqrt{2}\sqrt{N}} \sigma_{-1/2}(N)N^{1/30}.
% \end{align*}
% Collecting the constants in the first term gives $0.87$.
% Moreover, the usual argument shows that $\sigma_{-1/2}(N) \leq 2.66 N^{1/6}$.
% We obtain that the last expression is bounded by
% \[
% 1.74\frac{d}{\sqrt{N}} + 11.6\cdot 2.66 \cdot \sqrt{2} d N^{1/5-1/2} \leq 45.38 d N^{1/5-1/2}.
% \]

\subsection{Anisotropic finite quadratic modules}
\label{sec:anis-discr-forms}

A finite quadratic module $(A,Q)$ is called {\em isotropic}, if there exists an $x\in A\setminus \{0\}$ such that $Q(x)=0\in \Q/\Z$. Otherwise it is called {\em anisotropic}.
In this subsection
we now consider anisotropic finite quadratic modules. We show that there are only finitely many isomorphism classes of anisotropic finite quadratic modules $A$ for which $S_{k,A}$ is trivial. The following Lemma is a direct consequence of Theorem \ref{thm:jordan} and the theory of quadratic forms over finite fields.

\begin{lemma}
  \label{lem:anisoprim}
  Let  $(A,Q)$ be an anisotropic finite quadratic module of level $N$.
  Then $N=2^{t}N'$, where $N'$ is an odd square-free number and $t \in \{0,1,2,3\}$.
  If $p$ is a prime dividing $N$, then the $p$-component $A_p$ of $A$ belongs to the finite quadratic modules given in Table
  \ref{tab:anis}.
  \begin{table}[h]
    \centering
        \caption{The non-trivial isomorphism classes of anisotropic finite quadratic modules of prime-power order.
      The isomorphism classes of the finite quadratic modules in the last line depend only on the sum $s+t$.
      Here, $d(A)$ is the discriminant of $A$, equal to $|A| \in \Q^{\times}/(\Q^{\times})^2$. \label{tab:anis}}
    \begin{tabular}{@{}llll@{}}
      \toprule
      $p$                                      & genus symbol of $A$          & $\sig(A)$ & $d(A)$ \\
      \midrule
      \multirow{2}{*}{$p \equiv 1 \bmod{4}$}   & $p^{\pm 1}$            &  $4 + 2(1\pm\leg{p}{2})$ & $p$\\
      & $p^{-2}$                    & $4$ & $0$\\
      \midrule
      \multirow{2}{*}{$p \equiv 3 \bmod{4}$}   & $p^{\pm 1}$            &  $\pm 2 \leg{p}{2}$ & $p$ \\
      & $p^{+2}$                    & $4$ & $0$\\
      \midrule
      \multirow{4}{*}{$p=2$}                   & $2^{-2}$             & $4$ & $0$ \\
      & $2_{nt}^{\pm n}, t = 1,7, n = 1,2,3$ & $nt$  & $2^{n}$\\
      & $4_{t}^{\pm 1}, t = 1,3,5,7$ & $t$ & $0$\\
      & $2_{s}^{\pm 1}4_{t}^{\pm 1}, s=1,7, t=1,3,5,7$ & $s+t$ & $2$\\
      \bottomrule\vspace{0.5mm}
    \end{tabular}
  \end{table}
\end{lemma}

In the case of an anisotropic finite quadratic module it is also possible to obtain an explicit formula
for the quantity $\beta$ defined in \eqref{a5'} in terms of class numbers.
Before we can state the precise result we need to introduce some more notation.

Let $A$ be an anisotropic finite quadratic module of level $N$ and write $A = \bigoplus_{p \mid N} A_{p}$ for its
decomposition into $p$-components.
For each prime divisor $p$ of $N$, we denote the minimal number of generators of $A_p$ by $r_p$.
If $A_{p} = q^{\eps \cdot n}$ with $q = p^{r}$ we write $\eps_{A}(p) = \eps$.
We define a divisor $M$ of $N$ by
\[
  M = \prod_{\substack{p \mid N \text{ odd} \\ r_p = 2}} p
  \cdot
  \begin{cases}
    2, &\text{if } A_2 = 2^{-2},\\
    1, &\text{otherwise.}
  \end{cases}
\]

For $d \mid N$ we define the following auxiliary quantities:
\begin{align*}
  S_1(d)      &= \{ p \text{ prime};\, p \mid (d,M) \},\\
  \eps_{1}(d) &= (-1)^{|S_1(d)|},\\
  a_2(d)      &=
  \begin{cases}
    1,       & \text{if } r_2 = 0,\\
    2^{r_2-2}, & \text{if } r_2 > 0 \text{ and } d \text{ is odd},\\
    2^{\left[\frac{r_2}{2}\right]-1}, & \text{if } r_2 > 0 \text{ and } d \text{ is even,}
  \end{cases}
\end{align*}
%to allow for page break
\begin{align*}
  S_3(d)      &= \{p \text{ prime };\, p \mid \frac{d}{(d,M)}, p \equiv 3 \bmod{4}  \},\\
  \eps_{3}(d) &=
              \begin{cases}
                (-1)^{\frac{|S_3(d)|-1}{2}}, & \text{if } |S_{3}(d)| \text{ is odd,}\\
                (-1)^{\frac{|S_3(d)|}{2}}, & \text{if } |S_{3}(d)| \text{ is even.}
              \end{cases}
\end{align*}
For the $p$-components corresponding to odd primes, we define a sign
\[
  \epsodd(d) = \prod_{p \mid \frac{d}{(d,M)} \text{ odd}} \eps_A(p)\leg{p}{2}\leg{N/d}{p}.
\]
We let $N_2$ be the even part of $N$ and put $N_{2,d} = N_2 / (N_2,N/d)$,
$N_d = N/(d \cdot (N/d,N_2))$ and $d' = d/(d,M \cdot N_2)$.
Note that $d'$ is odd.
If $d$ is odd, we let $\eps_2(d) = 1$.
For even $d$, we define $\eps_2(d)$ in Table \ref{tab:eps2}.
\begin{table}[h]
  \centering
    \caption{$\eps_{2}(d)$ for even $d$.\label{tab:eps2}}
  \begin{tabular}[h]{lll}
    \toprule
    $A_{2}$ & $d' \equiv 1 \pmod{4}$ & $d' \equiv 3 \pmod{4}$\\
    \midrule
    $2^{-2}$ & $0$ & $1$ \\
    $2_{r_2 t}^{+r_2}, 4_{t}^{\pm 1}$ & $\leg{-N_{2,d}}{t N_d}$ & $\delta(r_2) \leg{N_{2,d}}{tN_d}$\\
    $2_1^{+1}4_{1}^{+1}$, $2_1^{+1}4_{3}^{-1}$ & $\leg{-8}{N/d}$ & $0$\\
    $2_1^{+1}4_{5}^{-1}$, $2_1^{+1}4_{7}^{+1}$ & $0$ & $\leg{8}{N/d}$\\
  \bottomrule
  \end{tabular}\vspace{1mm}
\end{table}
For simplicity, we also define $\eps(d) = \epsodd(d)\eps_{1}(d)\eps_{2}(d)\eps_{3}(d)$.

\begin{theorem}
  \label{a5'aniso}
  Let $(A,Q)$ be an anisotropic finite quadratic module of level $N$. We have
  \[
    \beta = -\sum_{\substack{d \mid N \\ d (d,M) \equiv 0,3 \bmod{4}}} \eps(d) a_2(d) \cdot (N/d,M) H(-d(d,M)).
  \]
  Here, $H(-n)$ is equal to the class number of primitive positive definite integral binary quadratic
  forms of discriminant $-n$ for $n > 4$ and $H(-3) = 1/3$ and $H(-4) = 1/2$.
\end{theorem}

\begin{proof}[Proof of Theorem \ref{a5'aniso}]
  Using the pointwise convergent
  Fourier expansion
  \begin{equation*}
    \B(x)=-\frac{1}{2\pi i} \sum_{n\in \Z-\{0\}} \frac{e(nx)}{n}
  \end{equation*}
  we find
  \begin{align*}
    \beta
    &= -\frac{1}{\pi} \sum_{n=1}^{\infty} \frac{1}{n} \Im(G(n,A))
    = -\frac{1}{\pi} \sum_{d \mid N} \sum_{\substack{n \geq 1 \\ (n,N/d) = 1}} \frac{1}{dn}\Im(G(dn,A)).
  \end{align*}
  First, we assume that $N$ is odd. For a discriminant $D$, we write $\chi_{D}$
  for the quadratic Dirichlet character modulo $\abs{D}$ given by $n \mapsto \leg{D}{n}$.
  Inserting the formula for $G(n,A)$ from Proposition~\ref{prop:gaussodd} and substituting $N/d$ for $d$, we obtain
  \begin{align*}
    \beta& = -\frac{\sqrt{N}}{\pi} \sum_{d \mid N} \sum_{\substack{n \geq 1 \\ (n,d) = 1}}
    \frac{(M,N/d)\sqrt{(M,d)}}{n\sqrt{N/d}}
    \Im\biggl(\prod_{\substack{p \mid d \\ r_p=1}} \eps_{A}(p) \leg{p}{2} \leg{nN/d}{p} e\left(\frac{1-p}{8}\right)
    \prod_{\substack{p \mid d \\ r_p=2}} (-1) \biggr)\\
             &= -\frac{\sqrt{N}}{\pi} \sum_{\substack{d \mid N \\ d(d,M) \equiv 3 \bmod{4}}}\eps(d)\cdot (M,N/d)\,
             \frac{\sqrt{(M,d)}}{\sqrt{N/d}} \sum_{\substack{n \geq 1 \\ (n,d) = 1}} \frac{\chi_{-d \cdot (M,d)}(n)}{n}\\
             &= - \sum_{\substack{d \mid N \\ d \cdot (d,M) \equiv 3 \bmod{4}}}\eps(d)\cdot (M,N/d)\,
             \frac{\sqrt{d \cdot (M,d)}}{\pi}\, L(\chi_{-d \cdot (M,d)},1).
  \end{align*}
  Here, we used that $e\left(\frac{1 - p}{8}\right) = \leg{p}{2}$ for $p \equiv 1 \pmod{4}$
  and $e\left(\frac{1 - p}{8}\right) = \leg{p}{2}i$ for $p \equiv 3 \pmod{4}$.
  Therefore, only divisors congruent to $3$ modulo $4$ contribute to the sum.
  Using that $L(\chi_{D},1)/\pi = H(D)/\sqrt{\abs{D}}$ (cf. \cite{Za}, Teil II, \S 8, Satz 5) for a negative discriminant $D$,
  we obtain the statement of the theorem in this case.

  If $N$ is even, we have to consider the different $2$-adic components separately.
  The case $A_2 = 2^{-2}$ is easy to obtain.
  We give a proof for $A_2 = 2_{rt}^{+r_2}$. The remaining cases
  are done analogously.  Using the same argument as before together with the results in Proposition \ref{prop:gausseven1},
  we obtain
  \begin{align*}
    \beta &= -\frac{\sqrt{N}}{\pi}
             \sum_{\substack{d \mid N\\ N/d \text{ odd}}}\epsodd(d)\eps_1(d)\cdot (M,N/d) \frac{\sqrt{(M,d)}}{\sqrt{N/d}} \\
             &\quad\quad \times
             \sum_{\substack{n \geq 1 \\ (n,d) = 1}} \frac{1}{n} \sqrt{2}^{r_2-2}
             \Im\biggl( \leg{tnN/d}{2}^{r_2} e\left(\frac{r_2tnN/d}{8}\right)
               \prod_{\substack{p \mid d \text{ odd}\\ r_p=1}}\leg{n}{p} \gamma_p\Biggr)\\
             &\quad\quad
             -\frac{1}{\pi} \sum_{\substack{d \mid N \\ N/d \equiv 0 \bmod{4}}}\eps(d)\cdot (M,N/d)
             \frac{\sqrt{N \cdot (M,d)}}{\sqrt{N/d}} \sum_{\substack{n \geq 1 \\ (n,d) = 1}} \frac{1}{n}\leg{-d \cdot (M,d)}{n}.
  \end{align*}
  Here, $\gamma_p = 1$ for $p \equiv 1 \bmod{4}$ and $\gamma_p = i$ for $p \equiv 3 \bmod{4}$.
  Using that $\sqrt{2}\leg{m}{2} e\left(\frac{m}{8}\right) = 1+\leg{-4}{m}i$, we obtain
  \begin{align*}
    \sqrt{2}^{r_2-2} &\Im\biggl( \leg{tnN/d}{2}^{r_2} e\left(\frac{r_2tnN/d}{8}\right)
               \prod_{\substack{p \mid d \text{ odd}\\ r_p=1}}\leg{n}{p} \gamma_p\biggr)\\
             &= a_2(N/d) \prod_{\substack{p \mid d \text{ odd}\\ r_p=1}}\leg{n}{p}
             \cdot
             \begin{cases}
               \leg{-4}{tnN/d}, & \text{if } d/(4 \cdot (M,d)) \equiv 1 \bmod 4\\
               \delta(r_{2})\leg{4}{tnN/d}, & \text{if } d/(4 \cdot (M,d)) \equiv 3 \bmod 4,
             \end{cases}
  \end{align*}
  which yields the statement of the theorem for $A_2 = 2_{nt}^{+r_{2}}$.
\end{proof}

The following upper bound for the class number is well known.
\begin{lemma} \label{lem:HD}
  Let $-D$ be a negative discriminant. We have
  \[
    H(-D) \leq \frac{\sqrt{D} \ln{D}}{\pi}.
  \]
\end{lemma}
\begin{proof}
  We use again $L(\chi_{D},1)/\pi = H(D)/\sqrt{\abs{D}}$
  and argue as in the proof of Lemma 5.6 on page 172 in \cite{vdg-hilbert}
  to obtain $L(\chi_{-D},1) \leq \ln{D}$ for $D>4$.
  Note that the bound for $H(-D)$ is also valid for $D=3$ and $D=4$ with our normalization
  that $H(-3) = 1/3$ and $H(-4) = 1/2$.
\end{proof}

\begin{lemma}
  \label{lem:a5'anisoest}
  Let $A$ be an anisotropic finite quadratic module.
  We have
  \[
    |\beta| \leq 1.71 \cdot |A|^{\frac{5}{8}} \ln(2|A|).
  \]
\end{lemma}
\begin{proof}
  We have by Theorem \ref{a5'aniso} and Lemma \ref{lem:HD} that
  \begin{align*}
    |\beta|
    &\leq \frac{1}{\pi}\sum_{\substack{d \mid N \\ d(d,M) \equiv 0,3 \pmod{4}}} a_2(d) (N/d,M) \sqrt{d(d,M)}\ln(d(d,M)) \\
    & \leq \frac{1}{\pi} M \ln(NM) c_2(A)
      \sum_{d \mid N} \sqrt{\frac{d}{(d,M)}},
  \end{align*}
  where $c_2(A) = 2$ if $r_{2} = 3$ and $c_2(A) = 1$, otherwise.
  We obtain
  \begin{equation*}
    |\beta| \leq \frac{1}{\pi}M \ln(NM) c_2(A) \sigma_0(M) \sigma_{1/2}(N/M).
  \end{equation*}
  If the order of $A$ is odd, we have $c_2(A) = 1$ and
  \begin{equation*}
    M \sigma_0(M) \sigma_{1/2}(N/M) = 1.76 \cdot |A|^{\frac{5}{8}}.
  \end{equation*}
  Moreover, if the order of $A$ is a power of $2$, then
  \begin{equation*}
    c_2(A) M \sigma_0(M) \sigma_{1/2}(N/M) \leq 3.05 |A|^{\frac{5}{8}}.
  \end{equation*}
  Using the multiplicativity of the divisor sum function, we see that
  \begin{equation*}
    c_2(A) M \sigma_0(M) \sigma_{1/2}(N/M) \leq 5.37 \cdot |A|^{\frac{5}{8}}.
  \end{equation*}
  Finally, using that $NM \leq 2|A|$ implies the statement of the lemma.
\end{proof}

\begin{corollary} \label{dimestaniso}
Let $(A,Q)$ be an anisotropic finite quadratic module.
% of level $N=2^tN'$ with $N'$ odd.
If $k\geq \frac{3}{2}$, then
\begin{align*}
\dim(S_{k,A})\geq  \frac{(|A|+1)(k-1)}{24} - 3.0 - 0.86 \cdot |A|^{\frac{5}{8}}\ln (2|A|).
\end{align*}
\end{corollary}

\begin{proof}
Since $A$ is anisotropic, we have $\alpha_4=1$.
Hence the estimate of Theorem \ref{fundest} can be refined to give
\begin{align*}
\left|\dim(S_{k,A})-\dim(M_{2-k,A(-1)})-\frac{d(k-1)}{12}\right|\leq R'(A),
\end{align*}
where
\begin{equation*}
R'(A) =\frac{\sqrt{|A[2]|}}{4} +\frac{1+\sqrt{|A[3]|}}{3\sqrt{3}} + \frac{1}{8}|A[2]| +\frac{1}{2} + \frac{\abs{\beta}}{2}.
\end{equation*}
Since $A$ is anisotropic, Lemma \ref{lem:anisoprim} implies that
 \begin{align*}
|A[2]|&\leq 8,\\
|A[3]| &\leq 9.
\end{align*}
Using in addition \eqref{eq:da} and the estimates $N'\leq |A|$ and $N\leq 2|A|$, we obtain
\begin{equation*}
\dim(S_{k,A}) \geq \frac{(|A|+1)(k-1)}{24}-\frac{\sqrt{2}}{2}-\frac{4}{3\sqrt{3}}-\frac{3}{2} - \frac{\abs{\beta}}{2}.
\end{equation*}
Together with Lemma \ref{lem:a5'anisoest} this proves the corollary.
\end{proof}

\begin{corollary}
  \label{cor:anisbound}
  Let  $(A,Q)$ be an anisotropic finite quadratic module such that $\sig(A)\equiv - 2k\pmod{4}$.
  If $|A|\geq 4.79 \cdot 10^7$, then $S_{\frac{3}{2},A}\neq \{0\}$
  and if $k \geq 2$, then $S_{k,A}\neq \{0\}$ for $|A| \geq 5.4 \cdot 10^6$.
\end{corollary}

We implemented the dimension formula and some of the estimates used here in python using \sage \cite{sage}.
Note that for the low weights $k=2$ and $k=\tfrac{3}{2}$, we need to calculate the dimension of the
invariants of the Weil representation. N. Skoruppa and S. Ehlen wrote a program
that determines the invariants explicitly and we included this implementation in our repository \cite{CodeRepo}.
Our complete software package, together with all required libraries, examples and documentation is available online \cite{CodeRepo}.

We used our program to obtain a list of all anisotropic finite quadratic modules
such that $S_{k,A} = \{0\}$ for $k \geq \tfrac{3}{2}$.

\begin{corollary}
  Let $(A,Q)$ be an anisotropic finite quadratic module such that $\sig(A)\equiv - 2k\pmod{4}$.
  Then $S_{k,A}= \{0\}$ for $k\geq \frac{3}{2}$ exactly if $A$ belongs to the lists given in Tables \ref{simpleanisotable32} and \ref{simpleanisotable}.
\end{corollary}

\begin{remark}
  The bound in Corollary \ref{cor:anisbound} can be improved substantially for higher weights.
  However, all of the bounds obtained this way are far away from the
  correct bounds (the maximal order is $238$ for $k = \tfrac{3}{2}$ and $60$ for $k=2$) found in Tables \ref{simpleanisotable32} and \ref{simpleanisotable}.
\end{remark}

\begin{table}[h]
  \caption{The 75 anisotropic finite quadratic modules $A$ with $S_{\frac{3}{2},A} = \{0\}$. Out of these, 59 have signature 1. \label{simpleanisotable32}}
  \begin{tabularx}{\linewidth}{l X} \toprule $\sig(A)$ & $\tfrac{3}{2}$-simple finite quadratic modules\\
    \midrule
    $1$ & $\left(\Z/2N\Z, \frac{x^{2}}{4N} \right)$ for $1 \leq N < 37$ square-free and \\
      & $N \in \{ 38, 39, 41, 42, 46, 47, 51, 55, 59, 62, 66, 69, 70, 71, 78, 87, 94, 95, 105, 110, 119 \},$ \\
      & $2_3^{+3}3^{+1},$ $2_5^{+3}3^{+2}$, $2_5^{+3}5^{+1}$, $4_7^{+1}3^{-1}$, $4_1^{+1}5^{-1}$, $4_5^{-1}5^{-2},$
      $4_3^{-1}7^{-1}$, $2_7^{+1}3^{+1}5^{-2}$, $2_7^{+1}3^{-1}5^{-1}$, $2_1^{+1}3^{+1}7^{+1}, 2_7^{+1}3^{+2}7^{-1}$,
      $4_1^{+1}13^{-1}$, $4_7^{+1}3^{+1}5^{+1},$ $2_7^{+1}5^{-1}7^{+1}$, $2_7^{+1}3^{+1}13^{+1}$\\
      \midrule
   $5$ & $4_5^{-1}$, $2_7^{+1}3^{+1}$, $2_5^{+3}$, $2_1^{+1}5^{+1}$, $4_7^{+1}3^{+1}$, $4_3^{-1}3^{-1}$, $2_7^{+1}7^{-1}$,
   $4_5^{-1}5^{-1}$, $4_1^{+1}5^{+1}$, $2_3^{+3}3^{-1}$, $2_1^{+1}13^{+1}$, $4_1^{+1}3^{+2}$, $4_7^{+1}11^{+1}$, $2_1^{+1}5^{-2}$,
   $2_3^{+3}7^{+1}$, $2_3^{+3}3^{-1}5^{-1}$\\\bottomrule
  \end{tabularx}\vspace{2mm}
\end{table}

\begin{table}[h]
  \caption{\label{simpleanisotable} Anisotropic finite quadratic modules $A$ with $S_{k,A} = \{0\}$ for $k \geq 2$.}
  \begin{tabularx}{\linewidth}{l l X} \toprule $k$ & $\sig(A)$ & genus symbols\\\midrule
    $2$ & 0 & $1^{+1}$, $5^{-1}$, $2_1^{+1}4_7^{+1}$, $3^{+1}11^{-1}$, $2^{-2}5^{+1}$, $2_2^{+2}3^{+1}$,
              $2_6^{+2}3^{-1}$, $13^{-1}$, $2_6^{+2}7^{+1}$, $17^{+1}$, $3^{-1}7^{-1}$, $2_1^{+1}4_1^{+1}3^{+1}$, $2_6^{+2}3^{+1}5^{+1}$\\
    $2$ & 4 & $2^{-2}$, $3^{+2}$, $5^{+1}$, $5^{-2}$, $2_1^{+1}4_3^{-1}$, $3^{-1}11^{-1}$, $2^{-2}5^{-1}$,
           $2_2^{+2}3^{-1}$, $2_6^{+2}3^{+1}$,
           $13^{+1}$, $2_2^{+2}7^{+1}$, $17^{-1}$, $3^{+1}7^{-1}$, $2_1^{+1}4_1^{+1}3^{-1}$, $2_2^{+2}3^{-1}5^{-1}$\\\midrule
    $\frac{5}{2}$ & 3 & $2_3^{+3}$, $4_3^{-1}$, $4_3^{-1}5^{-1}$, $2_1^{+1}3^{-1}$, $2_7^{+1}3^{+2}$, $2_1^{+1}7^{+1}$,
                  $2_1^{+1}11^{-1}$, $4_1^{+1}7^{+1}$, $2_7^{+1}5^{+1}$, $4_1^{+1}3^{-1}$, $4_5^{-1}3^{+1}$, $2_1^{+1}3^{-1}5^{-1}$\\
    $\frac{5}{2}$ & 7 & $2_7^{+1}$, $4_7^{+1}$, $2_1^{+1}3^{+1}$\\\bottomrule
  \end{tabularx}
 \begin{tabularx}{\linewidth}{l l X  l p{3cm}} \toprule $k$ & $\sig(A)$ & genus symbols & $\sig(A)$ & genus symbols\\\midrule
    $3$ & $2$ & $3^{-1}$, $2_2^{+2}$, $2^{-2}3^{+1}$, $7^{+1}$, $2_1^{+1}4_1^{+1}$, $11^{-1}$, $3^{+1}5^{+1}$, $3^{-1}5^{-1}$, $2_2^{+2}5^{-1}$, $23^{+1}$
        & $6$ & $3^{+1}$\\\midrule
    $\frac{7}{2}$ & $1$ & $2_1^{+1}$, $4_1^{+1}$, $2_7^{+1}3^{-1}$, $2_1^{+1}5^{-1}$, $4_3^{-1}3^{+1}$ & $5$ & $4_5^{-1}$\\\midrule
    $4$ & $0$ & $1^{+1}$, $5^{-1}$
        & $4$ & $5^{+1}$, $2^{-2}$\\\midrule
    $\frac{9}{2}$ & 3 & $4_3^{-1}$, $2_1^{+1}3^{-1}$
                  & 7 & $2_7^{+1}$\\\midrule
    $5$ & $2$ & $3^{-1}$, $2_2^{+2}$, $7^{+1}$
        & $6$ & $3^{+1}$\\\midrule
    $\frac{11}{2}$ & $1$ & $2_1^{+1}$, $4_1^{+1}$ & &\\\midrule
    $6$ & $0$ & $1^{+1}$ & &\\\midrule
    $7$ & $2$ & $3^{-1}$ & &\\\midrule
    $\frac{15}{2}$ & 1 & $2_1^{+1}$ & &\\\midrule
    $8$, $10$, $14$ & 0 & $1^{+1}$& &\\\bottomrule
  \end{tabularx}
\end{table}

\subsection{Differences of dimensions}
Here we give lower bounds for the difference of the dimensions of
$S_{k,A\oplus B}$ and $S_{k,A}$, where $A$ is an arbitrary
finite quadratic module and $B$ is an isotropic finite quadratic module of order
$p^2$ for large primes $p$.
We have to estimate the differences of the quantities occurring in the
dimension formula \eqref{dim1}. To indicate the dependency of the finite quadratic module,
we write $\alpha_i(A)$ for the quantities $\alpha_i$ associated to $A$ defined at the beginning of Section \ref{sect:de}.
We will make use of the following principle.
\begin{definition}
  If $A$ is a finite quadratic module and $U \subset A$ is a subgroup,
  we let
  \[
     U^{\perp} = \{ a \in A\ \mid\ (a,u) = 0 \text{ for all } u \in U \}
  \]
  be the \emph{orthogonal complement} of $U$.

  If $U$ is a \emph{totally isotropic subgroup}, that is, we have $Q(u) = 0$ for all
  $u \in U$, then the pair $(U^{\perp}/U,Q)$ also defines a finite quadratic module.
  We have $\abs{A} = \abs{U^{\perp}/U} \abs{U}^{2}$ and $\sig(A) = \sig(U^{\perp}/U)$.
\end{definition}

\begin{proposition}
  \label{up}
  Let $A$ be a finite quadratic module and let $B = U^{\perp}/U$ for
  some totally isotropic subgroup $U \subset A$.
  We have an injection $S_{k,B} \hookrightarrow S_{k,A}$
  given by $f \mapsto F$ with $F_{\alpha} = 0$  for $\alpha \not\in U^{\perp}$
  and $F_{\alpha} = f_{\alpha + U}$ for $\alpha \in U^{\perp}$.
\end{proposition}
\begin{proof}
  See Theorem 4.1 in \cite{Sch}.
\end{proof}

\begin{proposition}
  If $U \subset A$ is a maximal totally isotropic subgroup, then $A_0 = U^{\perp}/U$ is anisotropic.
  The isomorphism class of $A_0$ is independent of the choice of $U$
  and we call $A_{0}$ the \emph{anisotropic reduction} of $A$.
\end{proposition}
\begin{proof}
  It is easy to see that $U^{\perp}/U$ is anisotropic for a maximal totally isotropic subgroup:
  Suppose that $x \in U^{\perp}/U$ is isotropic, $x \neq 0$.
  Then $x = a + U$ for some $a \in U^{\perp}$ with $Q(a) = 0$.
  However, since $a \in U^{\perp}$ and $a \not\in U$, this implies that the subgroup $U'$ of $A$
  generated by $U$ and $a$ is isotropic and strictly larger than $U$.

  The uniqueness follows from the classification of the anisotropic finite quadratic modules
  (see Table \ref{tab:anis}) and the fact that $d(U^{\perp}/U) = d(A)$ and $\sig(U^{\perp}/U)=\sig(A)$.
\end{proof}

\begin{lemma}
  \label{lem:qdiff}
  Let $A$ be an arbitrary finite quadratic module,
  and let $B$ be an isotropic finite quadratic module of order $p^2$, where $p$ is a prime not dividing $6|A|$. Then
  \begin{align}
    \label{eq:qdiff}
    d_{A\oplus B}-d_A&=\frac{|A|}{2}(p^2-1),
  \end{align}
  and
  \begin{align*}
    |\alpha_1(A\oplus B)-\alpha_1(A)-\frac{|A|}{8}(p^2-1)|&\leq \frac{1}{2}\sqrt{|A[2]|},\\
    |\alpha_2(A\oplus B)-\alpha_2(A)-\frac{|A|}{6}(p^2-1)|&\leq \frac{2}{3\sqrt{3}}\left(1+\sqrt{|A[3]|}\right),\\
    |\alpha_3(A\oplus B)-\alpha_3(A)-\frac{|A|}{4}(p^2-1)|&\leq \frac{p-1}{2}|A|,\\
    % |\alpha_3(A\oplus B)-\alpha_3(A)+
    % \frac{\alpha_4(A\oplus B)-\alpha_4(A)}{2}-\frac{|A|}{4}(p^2-1)|&\leq ...,\\
    \alpha_4(A\oplus B)-\alpha_4(A)&\leq (p-1)|A|.
    % (p-1)\sqrt{|A|}\sigma(-1,A).
  \end{align*}
\end{lemma}

\begin{proof}
  To prove \eqref{eq:qdiff}, we use \eqref{eq:da}.
  % the fact that $d_A=\frac{|A|}{2}+\frac{|A[2]|}{2}$.
  Since $p\neq 2$, we have $B[2]=\{0\}$, and therefore
  \[
  d_{A\oplus B} = \frac{|A\oplus B|}{2}+\frac{|A[2]\oplus B[2]|}{2}= \frac{p^2|A|}{2}+\frac{|A[2]|}{2}.
  \]
  This implies the stated formula.

  The bounds for the differences of $\alpha_1$ and $\alpha_2$ directly follow from \eqref{eq:qdiff} and
  Lemma~\ref{a1a2est} combined with the fact that $B[3]=\{0\}$.

  Let $\kappa_0(A)$ denote the number of isotropic vectors in $A$.
  For $\alpha_4$ we use that
  \[
  \alpha_4(A)=\frac{1}{2}\left(\kappa_0(A) + \kappa_0(A[2])\right).
  \]
  Since the $2$-torsion of $B$ is trivial and since $p$ does not divide $|A|$, we find
  \begin{align*}
    \alpha_4(A\oplus B)-\alpha_4(A)&= \frac{1}{2}\left(\kappa_0(A\oplus B)-\kappa_0(A)\right)\\
    &= \frac{1}{2} \kappa_0(A)\left(\kappa_0(B) -1\right).
  \end{align*}
  The quantity $\kappa_0(B)$ is bounded by $2p-1$ by the Jordan decomposition (Theorem \ref{thm:jordan}),
  and $\kappa_0(A)$ is trivially bounded by $|A|$.
  % The proof of Lemma \ref{a4} shows that
  % \[
  % \kappa_0(A)\leq \sqrt{|A|}\sigma(-1,A),
  % \]
  This gives the claimed bound.

  We now turn to the estimate for $\alpha_3$. For $x\in \R$ we write $\{x\}=x-[x]\in [0,1)$ for the fractional part of of $x$.
  By definition we have
  \begin{align*}
    \alpha_3(A)&= \sum_{\gamma\in A/\{\pm 1\}} \{-Q(\gamma)\}\\
    &=\frac{1}{2}\sum_{\gamma\in A} \{-Q(\gamma)\}
    +\frac{1}{2}\sum_{\gamma\in A[2]} \{-Q(\gamma)\}.
  \end{align*}
  Consequently,
  \begin{align}
    \label{eq:a3diff}
    \alpha_3(A\oplus B)-\alpha_3(A)
    &=\frac{1}{2}\sum_{\gamma\in A\oplus B} \{-Q(\gamma)\}
    -\frac{1}{2}\sum_{\gamma\in A} \{-Q(\gamma)\}.
    % &=\frac{1}{2}\sum_{\substack{\gamma_1\in A\\ \gamma_2\in B\bs\{0\}}} \{-Q(\gamma_1)-Q(\gamma_2)\}
  \end{align}
  For an arbitrary $x\in \R$, we now estimate the sum
  \begin{align}
    \label{eq:bsum}
    S(x,B)= \sum_{\gamma\in B} \{x-Q(\gamma)\}.
  \end{align}

  If $B$ is the level $p$ isotropic finite quadratic module (which has genus symbol $p^{\eps \cdot 2}$ with $\eps = (-1)^\frac{p-1}{2}$),
  we have
  \begin{align*}
    S(x,B)&= \sum_{a,b\in \Z/p\Z} \{x-\frac{ab}{p}\}\\
    % &= (2p-1)\{x\} +(p-1) \sum_{b\in (\Z/p\Z)^\times} \{x+\frac{b}{p}\}\\
    &= p\{x\} +(p-1) \sum_{b\in \Z/p\Z} \{x+\frac{b}{p}\}\\
    &= p\{x\} +(p-1) \sum_{b=0}^{p-1} \left(\frac{1}{p}\{p x\}+\frac{b}{p}\right)\\
    &= p\{x\} +(p-1)\{p x\}+\frac{(p-1)^2}{2}
  \end{align*}
  Inserting this into \eqref{eq:a3diff}, we get
  \begin{align*}
    \alpha_3(A\oplus B)-\alpha_3(A)
    &=\frac{1}{2}(p-1)\sum_{\gamma\in A} \left(\{-Q(\gamma)\}+\{-p Q(\gamma)\}+\frac{(p-1)}{2}\right).
    % &\leq (p-1) |A| +\frac{(p-1)^2}{4}|A|.
  \end{align*}
  and therefore
  \begin{align}
    \label{eq:bsum2}
    |\alpha_3(A\oplus B)-\alpha_3(A)-\frac{|A|}{4}(p^2-1)|
    &\leq \frac{p-1}{2}|A|.
  \end{align}
  If $B$ is a finite quadratic module of level $q=p^2$,
  we slightly modify the above argument as follows.
  Let $\eps \in \Z$ with $\leg{2\eps}{p} = \pm 1$, such that
  $B$ has the genus symbol $q^{\pm 1}$.
  In this case we have
  \begin{align*}
    S(x,B)&= \sum_{a\in \Z/p^2\Z} \{x-\eps\frac{ a^2}{p^2}\}\\
    &= \sum_{a\in \Z/p\Z} \sum_{b\in \Z/p\Z} \{x-\eps\frac{ (a+pb)^2}{p^2}\}\\
    &= p\{x\} +\frac{(p-1)^2}{2}+\sum_{a\in (\Z/p\Z)^\times}\{p( x-\eps\frac{ a^2}{p^2})\}.
  \end{align*}
  By means of this identity, we obtain the same bound \eqref{eq:bsum2} as in the earlier case.
\end{proof}

\begin{theorem}
  \label{thm:qdiff}
  Let $A$ be an arbitrary finite quadratic module, and let $B$ be an isotropic finite quadratic module of order $p^2$, where $p$ is a prime not dividing $6|A|$.
  Then for $k \geq \tfrac{3}{2}$, we have
  \begin{align*}
    \dim(S_{k,A\oplus B})-\dim(S_{k,A}) &\geq \frac{|A|(p^2-1)}{24}\left( k-1 -\frac{36p}{p^2-1}\right). %\frac{|A|}{2}(p^2-1)
  \end{align*}
\end{theorem}

\begin{proof}
Let $U$ be any totally isotropic subgroup of $B \subset A \oplus B$. Then we have $\abs{B} = p$, $U^\perp = A \oplus U$ and $A \cong U^\perp/U$.
Therefore, Proposition~\ref{up} implies that the left hand side is non-negative.
We use the dimension formula
%\eqref{eq:dimcusp}
  \begin{align*}
    \dim(S_{k,A})&= \frac{d_A(k+12)}{12}-\alpha_1(A)-\alpha_2(A)-\alpha_3(A)-\alpha_4(A) +\dim(M_{2-k,A(-1)}).
    % \\
    % &= \frac{d(k-1)}{12} -(\alpha_1-\frac{d}{4}) -(\alpha_2-\frac{d}{3}) %-(\alpha_3-\frac{d}{2}+\frac{\alpha_4}{2})-  \frac{\alpha_4}{2}.
  \end{align*}
Because of Proposition~\ref{up}, we have $\dim(M_{2-k,(A\oplus B)(-1)})\geq \dim(M_{2-k,A(-1)})$.
Employing Lemma \ref{lem:qdiff}, we obtain
  \begin{align*}
    \dim(S_{k,A\oplus B})-\dim(S_{k,A})&\geq \frac{|A|(p^2-1)}{24}(k-1)-\frac{\sqrt{|A[2]|}}{2}-\frac{2+2\sqrt{|A[3]|}}{3\sqrt{3}}-\frac{3}{2}
    (p-1)|A|.
  \end{align*}
The claim now follows by the trivial estimates $|A[2]|\leq |A|$ and $\frac{2+2\sqrt{|A[3]|}}{3\sqrt{3}}\leq |A|$.
\end{proof}

\begin{corollary}
  \label{cor:pbounds}
  With the same assumptions as in Theorem \ref{thm:qdiff}, we have $S_{k,A \oplus B} \neq \{ 0 \}$ for $p \geq p_k$
  given in Table \ref{tab:pbounds}.
\end{corollary}
\begin{table}[h]
  \centering
    \caption{\label{tab:pbounds} Bounds on $p$ in Corollary \ref{cor:pbounds}.}
  \begin{tabular}{llllllllll}
    \toprule
    $k$   & $\frac{3}{2}$ & $2$ & $\frac{5}{2}$ & $3$ & $\frac{7}{2}$& $4$ & $\frac{9}{2} \leq k \leq \frac{13}{2}$& $7 \leq k \leq 9$ & $k \geq \frac{19}{2}$ \\\midrule
    $p_k$ &  73 & 37 & 29 & 19 & 17 & 13 & 11 & 7 & 5 \\\bottomrule
  \end{tabular}\vspace{3mm}
\end{table}

\section{Simple finite quadratic modules}
\label{sect:5}

If $A$ is a finite quadratic module and $k$ is an integer, we say that $A$ is $k$-simple
if $S_{k,A} = \{0\}$.
We will now develop an algorithm that allows us to easily iterate over all
finite quadratic modules starting from anisotropic ones.

For a finite quadratic module $A$ and an integer $n$, consider the finite set
of finite quadratic modules
\[
  B(A,n) = \{A' \ \mid\ A=U^\perp/U \text{ for a totally isotropic subgroup }
             U \subset A' \text{ with } \abs{U} = n \}.
\]
For simplicity, we define a subset $C(A,n) \subset B(A,n)$ using the following formal rules.
Let $p$ be an odd prime and $r \geq 0$ be an integer.
\begin{enumerate}
\item[(To)]  $1^{+1} \mapsto p^{\epsilon_p \cdot 2}$, where $\epsilon_p = 1$ if $p \equiv 1 \pmod{4}$ and $\epsilon_{p} = -1$, otherwise,
\item[(O)]   $(p^{r})^{\pm 1} \mapsto (p^{r+2})^{\pm 1}$.
\end{enumerate}
A rule can be applied to $A$ if the module on the left hand side of the rule
is a direct summand of $A$. It is important to note that we can always apply the rule
starting with the trivial module $1^{+1}$. We should also remark that if $r=0$ above,
the left hand side of the rule $(O)$ is the trivial module. Therefore,
for $r=0$, we have the rules $1^{+1} \mapsto (p^{2})^{+1}$ and $1^{+1} \mapsto (p^{2})^{-1}$.
The application of any of these rules to the genus symbol of $A$ yields the genus symbol of
a finite quadratic module in $B(A,p)$.

\begin{example}
  Consider the finite quadratic module $A$ given by the genus symbol $3^{+1}9^{-1}7^{-3}$.
  Applying rule (To) to $A$ for $p=3$, we obtain $3^{-3}9^{-1}7^{-3}$.
  Note that we can also apply rule (O) for both signs for $p=7$ by writing
  $3^{+1}9^{-1}7^{-3} = 3^{+1}9^{-1}7^{+2}7^{-1} \mapsto 3^{+1}9^{-1}7^{+2}343^{-1}$
  and similarly $3^{+1}9^{-1}7^{-3} = 3^{+1}9^{-1}7^{-2}7^{+1} \mapsto 3^{+1}9^{-1}7^{-2}343^{+1}$.
  By applying both rules once in all possible cases, we obtain
  \begin{align*}
    C(A,3) &= \{ 3^{-3}9^{-1}7^{-3}, 3^{+1}9^{-2}7^{-3}, 3^{+1}9^{+2}7^{-3}, 9^{-1}27^{+1}7^{-3},
                 3^{+1}81^{-1}7^{-3}
              \}\\
    C(A,7) &= \{ 3^{+1}9^{-1}7^{+5}, 3^{+1}9^{-1}7^{-3}49^{+1}, 3^{+1}9^{-1}7^{-3}49^{-1},
                 3^{+1}9^{-1}7^{-2}343^{+1}, 3^{+1}9^{-1}7^{+2}343^{-1} \}, \\
    C(A,p) &= \{ 3^{+1}9^{-1}7^{-3}p^{\epsilon_p \cdot 2}, 3^{+1}9^{-1}7^{-3}(p^{2})^{+1}, 3^{+1}9^{-1}7^{-3}(p^{2})^{-1}  \}
              \text{ for } p \not\in \{3, 7 \}.
  \end{align*}
\end{example}

For $p = 2$, the rules we require are more complicated.
Let $q = 2^{r}$ with $r \geq 1$.
\begin{enumerate}
\item[(Te1)]  $1^{+1} \mapsto 2^{+2}$,
\item[(Te2)]  $1^{+1} \mapsto 2_{0}^{+2}$,
\item[(E1)]   $q^{+2} \mapsto (2q)^{+2}$,
\item[(E2)]   $q_{4}^{-2} \mapsto (2q)^{-2}$,
\item[(E3)]   $q_{t}^{\pm 1} \mapsto (4q)_{t}^{\pm 1}$,
\item[(E4)]   $2^{+2} \mapsto 4_{0}^{+2}$,
\item[(E5)]   $2^{-2} \mapsto 4_{4}^{-2}$,
\item[(E6)]   $2_{2t}^{+2} \mapsto 4_{2t}^{+2}$ for $t \in \{1,7\}$,
\item[(E7)]   $2_{2t}^{+2} \mapsto 4_{2t}^{-2}$ for $t \in \{1,7\}$.
\end{enumerate}

\begin{remark}
  Note that $2_{0}^{+2} \cong 2_{4}^{-2}$ and therefore rule (E2) applies to $2_0^{+2}$, as well.
\end{remark}

\begin{definition}
  We define $C(A,p)$ to be the set of finite quadratic modules obtained from $A$
  after application of a single rule as listed above, only involving operations for $p$.
  For a prime power $n = p^{r}$, we define $C(A,p^{r})$ to be the set that is obtained
  from $r$ consecutive applications of rules only involving $p$.
  Finally, we define $C(A,n)$ for any positive integer $n$ by induction
  on the number of different primes dividing $n$ by putting
  \[
    C(A,p^{r}m) = \bigcup_{B \in C(A,m)} C(B,p^{r})
  \]
  for $(m,p)=1$.
\end{definition}

We use these formal rules because it is very easy to implement them on a computer.

\begin{theorem}\label{thm:rulesall}
  Let $A$ be a finite quadratic module and let $A_0$ be its anisotropic reduction.
  Then $A$ can be obtained from $A_0$ in finitely many steps
  using the rules given above. More precisely, we have $A \in C(A_0,n)$ for
  $n^2 = \abs{A}/\abs{A_0}$.
\end{theorem}
\begin{proof}
  It is enough to prove the claim for a $p$-module, that is a finite quadratic module
  of prime-power order $p^{n}$. Let us first assume that $p$ is odd and that $A$ has a genus symbol of the form
  $q^{\pm n}$ with $q = p^{r}$. Then it is easy to see that if $r$ is even, we can obtain the
  symbol $q^{\pm 1}$ starting from the trivial finite quadratic module
  \[
    1^{+1} \mapsto (p^{2})^{\pm 1} \mapsto \ldots \mapsto (p^{r})^{\pm 1}.
  \]
  Applying the same rule $n$ times, we obtain the symbol $q^{\pm n}$.
  If $r$ is odd instead, we start with the anisotropic symbol $p^{\pm 1}$.
  We obtain
  \[
    p^{\pm 1} \mapsto (p^{3})^{\pm 1} \mapsto \ldots \mapsto (p^{r})^{\pm 1}.
  \]
  We have now seen that we can obtain any finite quadratic $p$-module from a symbol
  of the form $p^{\pm n}$. Applying rule (To) several times
  reduces this symbol either to the trivial module or to the anisotropic
  finite quadratic module $p^{-\eps_{p}\cdot 2}$.

  For $p = 2$ we have to distinguish a few more cases.
  Suppose we are given a symbol of the form $q_{t}^{\pm r}$.
  We can obtain $q_{t}^{\pm r}$ from a symbol that is a direct
  sum of symbols of the form $2_{s}^{\pm r'}$ or $4_{s}^{\pm r'}$ by applying rule (E3).
  Using rules (E4-E7), any even number of odd summands of level $8$
  can be reduced to level $4$, leaving possibly a rank one odd component of level $8$.

  Now let $2_{t_{1}}^{+1} \ldots 2_{t_{r}}^{+1}$ be any odd discriminant form
  of level $4$ with $t_{1},\ldots,t_{r} \in \{1,7\}$.
  If $\{1,7\} \subset \{t_1,\ldots,t_{r}\}$, then $2_{0}^{+2} = 2_{1}^{+1}2_{7}^{+1}$
  is a summand.
  Now suppose that the rank is at least equal to four and the symbol does not
  contain both, $2_{1}^{+1}$ \emph{and} $2_{3}^{-1}$. Then $2_{4}^{+4}$ is a direct summand of $A$.
  However, $2_{4}^{+4} \cong 2^{-2}2_{4}^{-2} \cong 2^{-2}2_{0}^{+2}$.
  Then, we can apply (Te2) to reduce the rank of the level $4$ part to at most $3$.

  Finally, if we are given an even $2$-adic symbol $(2^{n})^{\pm r}$
  which is not anisotropic (i.e. is not $2^{-2}$), it always contains $(2^{n})^{+2}$
  or $(2^{n})^{-2}$ as a direct summand.
  Therefore, using rules (E1), (E2) and (E5), we can reduce to the case of level $2$ or $1$.
  Combining this with the strategy for the odd symbols gives the result.
\end{proof}

We now describe the algorithm used to compute all simple lattices.
\begin{algorithm}\label{algo}
  Given integers $r, s$ and a half-integer $k$, the following algorithm determines the
  isomorphism classes of all $k$-simple finite quadratic modules of signature $s$
  with a minimal number of generators less than or equal to $r$.
  \begin{enumerate}
  \item[(A)] Compute all anisotropic $k$-simple finite quadratic modules satisfying the conditions
             (see Table \ref{simpleanisotable}).
  \item[(B)] For each previously computed $k$-simple finite quadratic module $A$ compute the set $C(A,p)$
             for all primes $p \leq p_k$ with $p_k$ given in Table \ref{tab:pbounds}.
  \item[(C)] Repeat step (B) until no further $k$-simple finite quadratic modules have been found.
  \end{enumerate}
\end{algorithm}

The correctness of the algorithm follows from Theorem \ref{thm:rulesall} and Proposition \ref{up}
together with the results from the last sections. Moreover, that the algorithm terminates follows from Corollary \ref{cor:fin}.

\begin{remark}
  In each iteration, the bound on the primes can be reduced to
  the maximal prime such that there is a newly discovered finite quadratic module
  $A'$ in $C(A,p)$ for some $k$-simple finite quadratic module obtained one iteration earlier.
\end{remark}

The algorithm can be nicely illustrated in a graph. 
Figure \ref{fig:n=4} shows the output of the algorithm with the parameters corresponding to signature $(2,4)$.
In the graph, the $3$-simple finite quadratic modules have a red frame, the remaining ones a dotted green frame.
An edge from $A$ to some $B$ above $A$ indicates that $B$ is contained in $C(A,p)$ for a prime $p$. 
The color of the edge is the same as that of $A$. If an edge is green (and dotted), then the two modules connected by the edge
are both non-simple (and therefore colored green in the graph). Note that the graph does not contain all green edges for simplicity.
It contains at most one green ``incoming'' edge per vertex.

\begin{figure}
    \includegraphics[height=22cm]{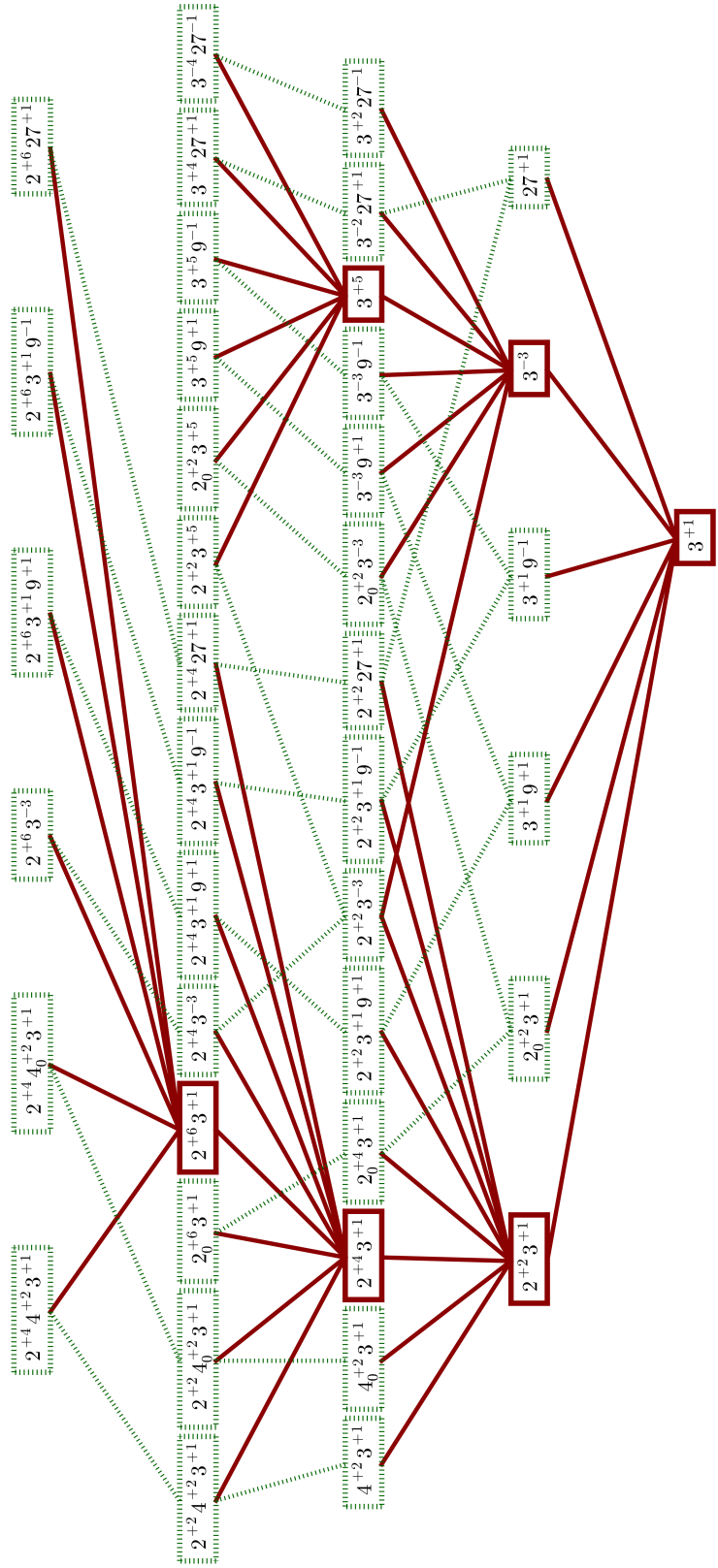}
  \caption{\label{fig:n=4}The graph illustrates Algorithm \ref{algo} for $k=3, r=6$ and $s=6$.}
  \end{figure}

In Tables \ref{tab:simple1}-\ref{tab:simple2}, we list all $\tfrac{2-n}{2}$-simple finite quadratic modules
of signature $2-n$ with minimal number of generators $r \leq 2+n$ for $n \geq 2$.

For $n = 1$, a large family of $\tfrac{3}{2}$-simple finite quadratic modules is given by the cyclic
finite quadratic modules $A_0(N) = (\Z/2N\Z,x^2/4N)$ for $N \in \Z_{>0}$. The corresponding orthogonal modular
varieties are the modular curves $\Gamma_0(N) \bs \h^*$,
where $\h^*=\h\cup \P^1(\Q)$. Moreover, each of these finite quadratic modules
has a global realization given by $L_{N} = \Z^3$ with quadratic form $Nx_1^2+x_2x_3$.
We have that $L_N$ is simple if and only if $1 \leq N \leq 36$ or if $N$ is in the following list:
\begin{gather*}
  38, 39, 40, 41, 42, 44, 45, 46, 47, 48, 49, 50, 51, 52, 54, 55, 56, 59, 60, 62, 63,\\
  64, 66, 68, 69, 70, 71, 72, 75, 76, 78, 80, 81, 84, 87, 90, 94, 95, 96, 98,\\
  100, 104, 105, 108, 110, 119, 120, 126, 132, 140, 144, 150, 168, 180.
\end{gather*} \label{Nlist}
The remaining finite quadratic modules for $n=1$ are included in \cite{CodeRepo}.

It is interesting to observe that for $N$ squarefree,
$L_N$ is simple if and only if $X_{0}^{\ast}(N) = \Gamma_0^{\ast}(N)\bs\h^*$ has genus zero.
Here, $\Gamma_{0}^{\ast}(N)$ is the extension of $\Gamma_{0}(N)$ by all Atkin-Lehner involutions.
The situation for non-squarefree $N$ is more complicated.
It would be interesting to find a similar geometric interpretation in the general case.

\begin{table}
\centering
\caption{The table shows the $70$ finite quadratic modules $A$ of signature $0$
  with minimal number of generators $r \leq 4$ and $S_{2,A} = \{0\}$. \label{tab:simple1}}
\begin{tabularx}{\linewidth}{lp{0.55\linewidth}|lp{0.45\linewidth}}
  \toprule level  & genus symbols & level & genus symbols\\\midrule
$1$ & $1^{+1}$ & $18$ & $2^{+2}9^{+1}$, $2^{+2}9^{-1}$, $2^{+4}9^{+1}$ \\
$2$ & $2^{+2}$, $2^{+4}$ & $20$ & $2_0^{+2}5^{-1}$ \\
$3$ & $3^{-2}$, $3^{+4}$ & $21$ & $3^{-1}7^{-1}$ \\
$4$ & $2_0^{+2}$, $4^{-2}$, $4^{+2}$, $2_0^{+4}$, $2^{+2}4^{+2}$, $[2^{+2}4^{-2}]$, $2_0^{+2}4^{+2}$, $[4^{-4}]$, $4^{+4}$ & $24$ & $2_1^{+1}4_1^{+1}3^{+1}$, $4_2^{-2}3^{+1}$\\
$5$ & $5^{-1}$, $5^{+2}$, $5^{-3}$, $5^{+4}$ & $25$ & $25^{+1}$, $25^{-1}$\\
$6$ & $2^{+2}3^{-2}$, $2^{+4}3^{-2}$, $2^{+2}3^{+4}$ & $27$ & $3^{+1}27^{-1}$ \\
$7$ & $7^{-2}$ & $28$ & $2_6^{+2}7^{+1}$\\
$8$ & $2_1^{+1}4_7^{+1}$, $4_0^{+2}$, $2_1^{+3}4_7^{+1}$, $2^{+2}4_0^{+2}$, $8^{+2}$, $2^{+2}8^{+2}$ & $32$ & $4_7^{+1}16_1^{+1}$\\
$9$ & $9^{+1}$, $9^{-1}$, $3^{-2}9^{-1}$, $9^{-2}$ & $33$ & $3^{+1}11^{-1}$\\
$10$ & $2^{-2}5^{+1}$, $2^{+2}5^{-1}$, $[2^{+4}5^{-1}]$, $2^{-4}5^{+1}$, $2^{+2}5^{+2}$ & $36$ & $2_0^{+2}9^{-1}$\\
$12$ & $2_2^{+2}3^{+1}$, $2_6^{+2}3^{-1}$, $2_0^{+2}3^{-2}$, $2_6^{+4}3^{-1}$, $2_6^{+2}3^{+3}$, $4^{+2}3^{-2}$ & $45$ & $9^{+1}5^{-1}$ \\
$13$ & $13^{-1}$ & $48$ & $2_3^{-1}8_3^{-1}3^{-1}$\\
$16$ & $2_7^{+1}8_1^{+1}$, $2_5^{-1}8_3^{-1}$, $4_7^{+1}8_1^{+1}$, $2_7^{+3}8_1^{+1}$, $8_0^{+2}$, $2_7^{+1}4^{+2}8_1^{+1}$ & $49$ & $49^{+1}$ \\
$17$ & $17^{+1}$ & $60$ & $2_6^{+2}3^{+1}5^{+1}$\\
     &          & $64$ & $2_7^{+1}32_1^{+1}$\\
\bottomrule
\end{tabularx}
\end{table}
\begin{table}[h!]
  \centering
  \caption{The table shows all finite quadratic modules of signature $2-n$ modulo $8$ with minimal number of generators
    $r \leq 2+n$ such that $S_{\frac{2+n}{2}} = \{0\}$\label{tab:simple2}}
  \begin{tabularx}{\linewidth}{lX}
    \toprule $n$ & \textbf{level}: genus symbols\\\midrule
    $3$  & $\mathbf{4}$ : $2_7^{+1}$, $2_7^{+3}$, $2_7^{+1}4^{+2}$, $2_7^{+5}$, $2_7^{+3}4^{+2}$, $2_7^{+1}4^{+4}$\\
         & $\mathbf{8}$  : $4_7^{+1}$, $2^{+2}4_7^{+1}$, $2^{+4}4_7^{+1}$, $\quad$
           $\mathbf{12}$ : $2_1^{+1}3^{+1}$, $2_7^{+1}3^{-2}$, $2_7^{+1}3^{+4}$,\\
         & $\mathbf{16}$ : $8_3^{-1}$, $8_7^{+1}$, $2^{+2}8_3^{-1}$, $[2^{+4}8_3^{-1}]$\\
    $4$  & $\mathbf{3}$ : $3^{+1}$, $3^{-3}$, $3^{+5}$, $\quad$
           $\mathbf{6}$ : $2^{+4}3^{+1}$, $[2^{+6}3^{+1}]$, $2^{+2}3^{+1}$\\
    $5$  & $\mathbf{8}$ : $4_5^{-1}$, $2^{+2}4_5^{-1}$, $2^{+4}4_5^{-1}$, $[2^{+6}4_5^{-1}]$\\
    $6$  & $\mathbf{2}$ : $2^{-2}, 2^{-4}, 2^{-6}$, $[2^{-8}]$, $\quad$
           $\mathbf{5}$ : $5^{+1}$\\
    $7$  & $\mathbf{8}$ : $4_3^{-1}$, $\quad$
         $\mathbf{12}$ : $2_1^{+1}3^{-1}$\\
    $8$  & $\mathbf{3}$ : $3^{-1}$, $\quad$
         $\mathbf{4}$ : $2_2^{+2}$, $\quad$
         $\mathbf{7}$ : $7^{+1}$\\
    $9$  & $\mathbf{4}$ : $2_1^{+1}$, $\quad$
         $\mathbf{8}$ : $4_1^{+1}$, $\quad$
         $\mathbf{16}$ : $8_1^{+1}$\\
    $10$ & $\mathbf{1}$ : $1^{+1}$, $\quad$
           $\mathbf{2}$ : $2^{+2}$\\
    $18,26$ & $\mathbf{1}$ : $1^{+1}$\\\bottomrule
  \end{tabularx}
\end{table}

\section{Simple lattices}
\label{sec:simple-lattices}

Let $L$ be an even lattice of signature $(2,n)$. We will say that $L$
is \emph{simple} if $S_{k,L'/L} = \{ 0 \}$ for $k = \frac{2+n}{2}$.
In this section we determine all isomorphism classes of simple lattices of signature $(2,n)$.
If $L$ is simple, then the finite quadratic module $L'/L$ is $k$-simple for $k = \frac{2+n}{2}$.
Thus, we are interested in all $k$-simple finite quadratic modules
with minimal number of generators $r \leq 2+n$ that actually correspond to a lattice of signature $(2,n)$.

\begin{proposition}
  \label{prop:glob}
  Let $A$ be a finite quadratic module and
  write $\eps_q$ for the sign of the Jordan component of $A$ of order $q$.
  Let $r_p$ be the minimal number of generators of $A_{p}$.
  There is an even lattice $L$ of signature $(r,s)$ with $L'/L = A$ if and only if all of the following conditions
  hold.
  \begin{enumerate}
  \item We have $\sig(A) \equiv r-s \pmod{8}$.
  \item For all primes $p$, we have $r+s \geq r_{p}$.
  \item For all odd primes $p$ with $r+s = r_p$, write $(-1)^s\abs{A} = p^{\alpha}a_p$ with $(a_p,p) = 1$.
     Then we have
    \begin{equation}
      \label{eq:signs}
      \prod_{q} \eps_q = \leg{a}{p},
    \end{equation}
  where the product runs over all powers $q$ of $p$.
  \item If $r+s = r_2$ and $A_2$ does not contain an odd direct summand of the form $2_{t}^{\pm m}$ with $m \geq 1$,
    then \eqref{eq:signs} holds for $p=2$ and $(-1)^s\abs{A} = 2^{\alpha}a$ with $(a,2) = 1$, as well.
  \end{enumerate}
\end{proposition}
\begin{proof}
  See \cite{Ni}, Theorem 1.10.1.
\end{proof}

Using Proposition \ref{prop:glob}, we determined all genus symbols
that do not correspond to lattices of signature $(2,n)$.
We enclosed them in parentheses $[ \cdot ]$ in Tables \ref{tab:simple1}--\ref{tab:simple2}.
For instance, the module $M=2^{+2}4^{-2}$ in Table \ref{tab:simple1}
does not correspond to any lattice of signature $(2,2)$ because the product of the signs
is equal to $-1$ but $a=1$ and $M$ does not contain an odd component of level $4$.
It does, however, correspond to a (non-simple) lattice of signature $(2,2+8m)$ for all $m\geq 1$.

Recall that the genus symbol does only determine a lattice up to rational equivalence.
A genus consists of finitely many integral isometry classes of lattices.
However, the following proposition gives a full classification of all isomorphism classes of
lattices of signature $(2,n)$ for $n \geq 1$.

\begin{proposition}
  If $L$ is a simple lattice of signature $(2,n)$ and $L$ is not contained in the genus
  $2_1^{+1}5^{-1}25^{-1}$, then its genus contains a unique isomorphism class.
  The genus $2_1^{+1}5^{-1}25^{-1}$ contains two isomorphism classes.
\end{proposition}
\begin{proof}
  For $n \geq 2$,
  Corollary 22 in Chapter 15 of \cite{CS} states that
  if there is more than one class in the genus of an indefinite lattice $L$,
  then $\abs{\det(L)} \geq 5^{6}$. There is no finite quadratic module of this size in our list
  in Tables \ref{tab:simple1}-\ref{tab:simple2}.

  The lattices in signature $(2,1)$ are slightly more complicated to treat.
  Using Theorem 21 in Chapter 15 of \cite{CS}, we find that only the lattices of discriminant $d$
  with $4d$ divisible by $5^{3}$ or $8^{3}$ might contain more than one class
  in their genus. Theorem 19 ibid. finally leaves us with the following list
  of genera that might contain more than one class:
  \begin{equation*}
    2_7^{+1} 4_1^{+1} 32_1^{+1}, \quad 2_1^{+1} 4_7^{+1} 16_1^{+1}, \quad
    2_7^{+1} 8_1^{+1} 16_1^{+1}, \quad 2_1^{+1}5^{-1}25^{-1}.
  \end{equation*}
  For the first $3$ genera, we can use \cite{EarnestHsia}, Theorem 2.2, to see
  that these also only contain one class.
  The last one is treated in \cite{Watson} in Chapter 7, Section 5.
  This genus contains two isomorphism classes represented by the integral ternary quadratic forms
  \begin{align*}
    \phi_1 (x_1, x_2, x_3) &= x_1^2+x_1x_2-x_2^2+25x_3^2,\\
    \phi_2 (x_1, x_2, x_3) &= 5(x_1^2+x_1x_2-x_2^2) + x_3^2.
  \end{align*}
\end{proof}

\subsection{Applications to Borcherds products}
\label{sect:bps}

Let $(L,Q)$ be an even lattice of signature $(2,n)$ and let $A=L'/L$ be its the discriminant module.
We write $\Orth(L)$ for the orthogonal group of $L$ and  $\Orth(L)^+$ for the subgroup of index $2$
consisting of those elements whose determinant has the same sign as the spinor norm.
We consider the kernel $\Gamma_L$ of the natural homomorphism $ \Orth(L)^+\to \Aut(A)$, sometimes
referred to as the stable orthogonal group of $L$.

Let $D$ be the hermitian symmetric space associated to the group $\Orth(L\otimes_\Z \R)$. It can be realized as a tube domain in $\C^n$. The group $\Gamma_L$ acts on $D$ and the quotient
\[
X_L = \Gamma_L\bs D
\]
has the structure of a quasi-projective algebraic variety.
%It is projective if and only if $L$ is anisotropic.
For suitable choices of $L$, important families of classical modular varieties can be obtained in this way, including Shimura curves, Hilbert modular surfaces and Siegel modular threefolds.

For every $\mu\in A$ and every negative $m\in \frac{1}{N}\Z$, there is a Heegner divisor $Z(m,\mu)$ on $X_L$ (sometimes also referred to as special divisor or rational quadratic divisor), see
e.g.~\cite{Bo2}, \cite{Br1}. We denote by $\Pic_{\mathrm{Heeg}}(X_L)$ the subgroup of the Picard group $\Pic(X_L)$ of $X_L$ generated by all such Heegner divisors.

If $L$ is simple, then for {\em every} pair $(m,\mu)$ as above there is a weakly holomorphic modular form $f_{m,\mu}\in M^!_{1-n/2,A(-1)}$ of weight $1-n/2$ whose Borcherds lift $\Psi(f_{m,\mu})$ (in the sense of Theorem 13.3 in \cite{Bo1}) is a meromorphic modular form for $\Gamma_L$
%with some finite order multiplier system
whose divisor is $Z(m,\mu)$.
In particular, the vector space $\Pic_{\mathrm{Heeg}}(X_L)\otimes_\Z\Q$ is one-dimensional and generated by the Hodge bundle, hence it is as small as it can be.

The weight of the Borcherds product
$\Psi(f_{m,\mu})$  is given by half of the constant term of the component of $f_{m,\mu}$
corresponding to the characteristic function $\phi_0$ of the zero element of $A$. Equivalently, it can be expressed in terms of the coefficient of index $(-m,\mu)$ of the unique normalized Eisenstein series $E_{1+n/2,A}\in M_{1+n/2,A}$ whose constant term is $\phi_0$, see e.g. Theorem~12 of \cite{BK}. The coefficients of such Eisenstein series can be explicitly computed, see Theorem~7 of \cite{BK} or \cite{KY}.
It would be interesting to use the list of simple lattices to search systematically for holomorphic Borcherds products of singular weight $n/2-1$ for $\Gamma_L$. Such Borcherds products are often denominator identities of generalized Kac-Moody algebras, see \cite{Sch-Inv}.


\begin{thebibliography}{GeNy}
\newcommand{\etalchar}[1]{$^{#1}$}

\bibitem[AS]{AS} {\em M. Abramowitz and I. Stegun}, Pocketbook of Mathematical Functions, Verlag Harri Deutsch, Thun (1984).

%\bibitem[Bo1]{Bo}  \emph{R. E. Borcherds}, Automorphic forms on $\Orth_{s+2,2}(\R)$
%     and infinite products, Invent. Math. {\bf 120} (1995), 161--213.

\bibitem[Bo1]{Bo1}
  \emph{R. E. Borcherds}, Automorphic forms with singularities on
  Grassmannians, Invent. Math. \textbf{132} (1998), 491--562.

\bibitem[Bo2]{Bo2}
  \emph{R. Borcherds}, The Gross-Kohnen-Zagier theorem in higher
  dimensions, Duke Math. J. \textbf{97} (1999), 219--233.
  Correction in: Duke Math J. \textbf{105} (2000), 183--184.

\bibitem[Bo3]{Bo3} {\em R. E. Borcherds}, Reflection groups of
  Lorentzian lattices, Duke Math. J. {\bf 104} (2000), 319--366.

\bibitem[Br1]{Br1} \emph{J. H. Bruinier}, Borcherds products on
  $\Orth(2,l)$ and Chern classes of Heegner divisors, Springer Lecture
  Notes in Mathematics {\bf 1780}, Springer-Verlag (2002).

\bibitem[Br2]{Br:Pic} \emph{J. H. Bruinier}, On the rank of Picard groups of modular varieties attached to orthogonal groups, Compos. Math. {\bf 133} (2002), 49--63.

  % \bibitem[BF]{BF} {\em J. H. Bruinier and J. Funke}, On two geometric
  %   theta lifts, Duke Math. Journal. {\bf 125} (2004), 45--90.

\bibitem[BK]{BK} \emph{J. H. Bruinier and M. Kuss}, Eisenstein series attached to lattices and modular forms on orthogonal groups, Manuscr. Math. {\bf 106} (2001), 443--459.

\bibitem[Bu]{Bu} {\em M. Bunschuh}, \"Uber die Endlichkeit der Klassenzahl gerader Gitter der Signatur $(2,n)$ mit einfachem Kontrollraum, Dissertation universit\"at Heidelberg (2002).

\bibitem[CS]{CS} {\em J. H. Conway and N. J. Sloane}, Sphere packings, lattices and groups. Third edition.
%With additional contributions by E. Bannai, R. E. Borcherds, J. Leech, S. P. Norton, A. M. Odlyzko, R. A. Parker, L. Queen and B. B. Venkov.
Grundlehren der Mathematischen Wissenschaften [Fundamental Principles of Mathematical Sciences], 290. Springer-Verlag, New York (1999).

\bibitem[EH]{EarnestHsia} {\em A. G. Earnest and J. S. Hsia}, Spinor norms of local integral rotations. {II},
  Pacific J. Math. {\bf 61} (1975), 71--86.

\bibitem[Ehl]{CodeRepo} {\em S. Ehlen}, Finite quadratic modules and simple lattices, Source code and resources (2014).\\
  \newblock{\url{http://www.github.com/sehlen/sfqm}}.

\bibitem[Fi]{Fi} {\em J. Fischer}, An approach to the Selberg trace
  formula via the Selberg zeta-function, Lecture Notes in Mathematics
  {\bf 1253}, Springer-Verlag (1987).

\bibitem[Fr]{Fr} {\em E. Freitag}, Riemann surfaces, manuscript (2013).

\bibitem[Ge]{vdg-hilbert}
{\em G. van~der Geer}, Hilbert modular surfaces, Ergebnisse der Mathematik
  und ihrer Grenzgebiete (3) [Results in Mathematics and Related Areas (3)],
  vol.~16, Springer-Verlag, Berlin (1988).

\bibitem[Ha]{Ha} {\em H. Hagemeier}, Automorphe Produkte singul\"aren Gewichts, Dissertation, Technische Universit\"at Darmstadt (2010).

\bibitem[Ki]{Ki} {\em Y. Kitaoka},
  Arithmetic of quadratic forms,Cambridge Tracts in Mathematics  {\bf 106},
  Cambridge University Press (1993).

\bibitem[KY]{KY} {\em S. Kudla and T. Yang},  Eisenstein series for $SL(2)$,  Sci. China Math. {\bf 53} (2010),  2275�-2316.

\bibitem[Ni]{Ni} {\em V. V. Nikulin},
  Integer symmetric bilinear forms and some of their geometric
  applications (Russian), Izv. Akad. Nauk SSSR Ser. Mat. {\bf 43}
  (1979),  111--177. English translation: Math USSR-Izv. {\bf
    14} (1980), 103--167.

\bibitem[No]{No} {\em A. Nobs}, Die irreduziblen Darstellungen der
  Gruppen $\Sl_2(\Z_p)$, insbesondere $\Sl_2(\Z_2)$. I. Teil,
  Comment. Math. Helvetici {\bf 51} (1976), 465--489.

\bibitem[Sch1]{Sch-Inv} {\em N. Scheithauer}, On the classification of automorphic products and generalized Kac-Moody algebras, Invent. Math. {\bf 164} (2006), 641--678.

\bibitem[Sch2]{Sch} {\em N. Scheithauer}, Some constructions of modular forms for the Weil representation of $\Sl_2(\Z)$, preprint (2011).

\bibitem[Sk1]{Sk1} {\em N.-P. Skoruppa}, \"Uber den Zusammenhang zwischen Jacobiformen und Modulformen halbganzen Gewichts, Bonner  Mathematische Schriften {\bf 159} (1985).

\bibitem[Sk2]{Sk2} {\em N.-P. Skoruppa}, Jacobi forms of critical weight and Weil
representations. In: Modular Forms on Schiermonnikoog (Eds.: B.~Edixhoven et.al.), Cambridge Univerity Press (2008), 239--266.

\bibitem[Sk3]{Sk3} {\em N.-P. Skoruppa}, Finite Quadratic Modules and Weil representations, in preparation.

\bibitem[S{\etalchar{+}}14]{sage}
W.\thinspace{}A. Stein et~al.
\newblock {\em {S}age {M}athematics {S}oftware ({V}ersion 6.2)}.
\newblock The Sage Development Team, 2014.
\newblock {\url{http://www.sagemath.org}}.

\bibitem[Str]{Str} {\em F. Str\"{o}mberg}, Weil representations associated with finite quadratic modules. Mathematische Zeitschrift {\bf 275} (2013), 509--527.

\bibitem[Wa]{Watson} {\em G. L. Watson}, Integral Quadratic Forms, Cambridge University Press (1960).

\bibitem[Za]{Za} {\em D. Zagier}, Zetafunktionen und quadratische K\"orper, Springer-Verlag (1981).
\end{thebibliography}
\end{document}